\documentclass[11pt]{article} 
\usepackage{enumitem}
\setlist[enumerate]{label=(\roman*)} 

\usepackage{tabularx} 
\usepackage{amssymb}
\usepackage{amsmath}
\usepackage{amsthm}
\usepackage{latexsym}
\usepackage{amsfonts}
\usepackage{geometry}
\usepackage[hyphens]{url} 
\usepackage{hyperref}

\usepackage{mathtools}

\usepackage{bbm} 

\numberwithin{equation}{section} 

\newtheorem{Theorem}{Theorem}
\theoremstyle{definition}
\newtheorem{Lemma}[Theorem]{Lemma}
\newtheorem{Definition}[Theorem]{Definition}

\newtheorem{Corollary}[Theorem]{Corollary}
\newtheorem{Remark}[Theorem]{Remark}
\newtheorem{Proposition}[Theorem]{Proposition}

\usepackage{mathtools}
\DeclarePairedDelimiter\abs{\lvert}{\rvert}%
\DeclarePairedDelimiter\norm{\lVert}{\rVert}%

\makeatletter
\let\oldabs\abs
\def\abs{\@ifstar{\oldabs}{\oldabs*}}
\let\oldnorm\norm
\def\norm{\@ifstar{\oldnorm}{\oldnorm*}}
\makeatother

\usepackage{mathrsfs}

\newcommand{\R}{\mathbb{R}}
\newcommand{\Rp}{\mathbb{R}_{+}}

\newcommand{\N}{\mathbb{N}}
\newcommand{\Q}{\mathbb{Q}}
\newcommand{\C}{\mathbb{C}}

\newcommand{\ud}{\mathrm{d}}

\newcommand{\ds}{\, \mathrm{d}s}
\newcommand{\dt}{\, \mathrm{d}t}
\newcommand{\dr}{\, \mathrm{d}r}

\newcommand{\dL}{\, \mathrm{d}L}

\newcommand{\dX}{\, \mathrm{d}X}

\newcommand{\E}{\mathbb{E}}

\newcommand{\1}{\mathbbm{1}} 
\newcommand{\Borel}{\mathfrak{B}}

\newcommand{\scapro}[2]{\langle #1,#2\rangle}    
\newcommand{\HSnorm}[1]{\norm{#1}_{L_{\rm HS}(U,H)}}

\DeclarePairedDelimiter\bignorm{\big\Vert}{\big\Vert}
\DeclarePairedDelimiter\smallnorm{\Vert}{\Vert}

\DeclareMathOperator{\e}{\mathrm{e}}

\let\phi\varphi
\let\epsilon\varepsilon

\DeclareMathOperator{\Span}{\mathrm{Span}}

\DeclareMathOperator{\Id}{\mathrm{Id}}

\allowdisplaybreaks

\title{Stochastic evolution equations  driven by \\  cylindrical stable noise}

\author{Tomasz Kosmala \thanks{t.kosmala@qmul.ac.uk} \\ School of Mathematical Sciences \\ Queen Mary University of London \\ Mile End Road, London E1 4NS\\ United Kingdom \and Markus Riedle \thanks{markus.riedle@kcl.ac.uk} \\Department of Mathematics \\ King's College London\\ London WC2R 2LS\\ United Kingdom}

\usepackage{letltxmacro}
\LetLtxMacro\Oldfootnote\footnote

\usepackage{todonotes} 

\begin{document}
\date{August 2, 2021}
\maketitle

\begin{abstract}
We prove existence and uniqueness of a mild solution of a stochastic evolution equation driven by a standard $\alpha$-stable cylindrical L\'evy process defined on a Hilbert space for $\alpha \in (1,2)$. The coefficients are assumed to map between certain domains of fractional powers of the generator present in the equation. The solution is constructed as a weak limit of the Picard iteration using tightness arguments. Existence of strong solution is obtained by a general version of the Yamada--Watanabe theorem.
\end{abstract}





\section{Introduction}

Standard symmetric $\alpha$-stable distributions are the natural generalisations of Gaussian distributions for modelling random  perturbations of all kinds of dynamical systems. In the mathematical perspective, they are analytically tractable and well understood, and in the applied perspective they often meet various empirical requests, such as heavy tails, self-similarity and infinite variance. The importance of these models is reflected by the available vast literature on dynamical systems perturbed by random noises with $\alpha$-stable distributions. 

Surprisingly, there are only a few results known for partial differential equations perturbed by  $\alpha$-stable distributions. In fact, only in the random field approach, based on the seminal work by Walsh, one can find publications on stochastic partial differential equations (SPDEs) driven by multiplicative $\alpha$-stable noise, whereas in the semigroup approach, in the spirit of Da Prato and Zabczyk, one can only find results for equations with additive driving noise  distributed according to an $\alpha$-stable law. SPDEs driven by L\'evy processes in the semigroup approach are considered in the monograph by Peszat and Zabczyk \cite{Peszat_Zabczyk}; but here, if the driving noise is multiplicative, the L\'evy process is assumed to be a genuine Hilbert space-valued process. The lack of results in the semigroup approach is due to the fact that a random noise with a standard $\alpha$-stable distribution does not exist as an ordinary  Hilbert space-valued process but only in the generalised sense of  Gel'fand and Vilenkin \cite{Gelfand_Vilenkin} or  Segal \cite{Segal}. 

The purpose of this work is to close this gap and to provide the first existence result for a general evolution equation of the form
\begin{equation}
\label{intro_stochastic_evolution_equation}
\dX(t) =  \big( AX(t) + F(X(t)) \big) \dt + G(X(t)) \dL(t),\qquad t\in [0,T].
\end{equation}
Here, $A$ is the generator of a strongly continuous semigroup on a Hilbert space $H$, the non-linearity is described by the mapping $F\colon H\to H$, and the diffusion operator by $G\colon H\to L_{\rm HS}(U,H)$. 
The noise $L$ is modelled by a generalised process with a standard symmetric $\alpha$-stable distribution for $\alpha \in (1,2)$, and as such it can be interpreted as a specific example of a cylindrical L\'evy process in the framework recently developed by Riedle and co-authors. The precise conditions on the coefficients can be found in Theorem \ref{th_existence_weak_sol_main_result}.

In the random field approach, SPDEs with stable noise were considered the first time by Mueller \cite{Mueller} for $\alpha \leq 1$, and  Mytnik \cite{Mytnik} for $1<\alpha<2$ in the context of certain equations with non-Lipschitz coefficients; see also Xiong, Yang and Zhou \cite{Xiong_Yang,Yang_Zhou}.
Existence and uniqueness of solutions for equations with Lipschitz coefficients were recently proved by Balan \cite{Balan}. Her results were improved by Chong \cite{Chong_stochastic_PDEs} and by  Chong, Dalang and Humeau \cite{Chong_Dalang_Humeau}, who additionally characterised the path space of solutions as negative Sobolev spaces. 

As mentioned above, in the semigroup approach only equations with additive noise were considered. This was mainly accomplished in the work \cite{Brzezniak_Zabczyk} by Brze\'zniak and Zabczyk  on existence and regularity of solutions by modelling the stable noise as a subordinated cylindrical Brownian motion. In the setting of generalised processes, the linear equation driven by a standard $\alpha$-stable process was considered by Riedle \cite{Riedle_stable}.

The reason for the lack of results for SPDEs driven by a multiplicative $\alpha$-stable noise is due to the generalised form of such processes not attaining values in the underlying Hilbert space. This in particular implies the lack of a L\'evy-It{\^o} decomposition, which is usually the very foundation for a theory of stochastic integration and for the derivation of the existence of a solution for stochastic differential equations. A theory of stochastic integration for the class of cylindrical L\'evy processes has been introduced in Jakubowski and Riedle \cite{Jakubowski_Riedle} by arguments avoiding the usual   L\'evy-It{\^o} decomposition. In the current article we continue this line of research and establish the existence of solution for \eqref{intro_stochastic_evolution_equation} without utilising a L\'evy-It{\^o} decomposition. For these reasons our arguments differ from approaches of other publications on similar results in comparable settings, and thus we believe it is worth to highlight the main steps in the following. 

Due to the lack of a L\'evy-It{\^o} decomposition,  we are bound to establish convergence of the Picard iteration in one piece, i.e.\ without utilising a decomposition of the driving noise into small and large jumps.
We succeed in establishing the existence result by  first showing that, if the time interval is sufficiently small, solutions of \eqref{intro_stochastic_evolution_equation} are pathwise unique, which follows from a version of Gronwall's inequality due to Willet and Wong \cite{Willett_Wong}. This step is followed by constructing a solution, mild in the analytical and weak in the probabilistic sense,  as a limit of the Picard iteration. For this purpose, we 
establish tightness of the Picard approximation by  Aldous' condition and a special version of the compact containment condition. This analysis is based on some fractional calculus with similar estimates as in Hausenblas \cite{Hausenblas}.  The previous step  enables us to conclude 
the existence of a solution as the  almost sure limit of the Picard iteration on another probability space by applying Skorokhod's theorem. The limit is identified as a solution of \eqref{intro_stochastic_evolution_equation} by a careful analysis of the increments of the Picard iteration. A general version of the Yamada--Watanabe theorem enables us to conclude the existence of solutions on any given probability space i.e.\ existence of the strong solution, which finally results in the existence of solution for \eqref{intro_stochastic_evolution_equation} on any time interval by gluing solutions on smaller intervals together. 

This approach is based on some distinctive arguments or methods highlighted in the following: 
\begin{itemize}
\item[(a)] The standard $\alpha$-stable process does not exist as a Hilbert space-valued stochastic process similar as the standard cylindrical Brownian motion does not.  But in the latter case, finite second moments and Gaussian distribution enable a rather straightforward calculus of  stochastic integration. In the setting of this work, we rely on the stochastic integration theory for cylindrical L\'evy processes developed in Jakubowski and Riedle \cite{Jakubowski_Riedle}. 
We obtain estimates for those stochastic integrals with respect to a standard $\alpha$-stable cylindrical L\'evy processes by generalising the arguments in Gin\'e and Marcus \cite{Gine_Marcus} in the finite-dimensional setting. To be precise, we show a bound on the tails of the integral
\begin{equation*}
\sup_{r>0} r^\alpha P\left( \sup_{t \in [0,T]} \norm{\int_0^t \Psi(s) \dL(s)} >r \right) 
\leq c_{2,\alpha} \E \left[ \int_0^T \HSnorm{\Psi(s)}^\alpha \ds \right].
\end{equation*}

\item[(b)] The integral operator $I(\Psi)=\int_0^T \Psi(s)\dL(s)$ for an $\alpha$-stable L\'evy process $L$ and an admissible stochastic process $\Psi$ only maps continuously from an $L^\alpha$-space of integrands to the $L^p$-space of random variables for $p\lneqq\alpha$; see Rosi\'nski and Woyczy\'nski \cite{Rosinski_Woyczynski_moment} or the inequality
$$\E \left[ \sup_{t \in [0,T]} \norm{\int_0^t \Psi(s) \dL(s)}^p \right]
\leq C_{\alpha,p} \left( \E \left[ \int_0^T \HSnorm{\Psi(s)}^\alpha \ds \right] \right)^{p/\alpha}, \qquad p< \alpha,$$
which we prove in Corollary \ref{pro_moment_inequality_2}.
Thus, since the domain of the integral operator is smaller than its range, we cannot apply a fixed point theorem. 
As a consequence, instead of working with moments, we derive convergence of the Picard iteration by tightness arguments in the Skorokhod space using a version of Aldous' condition and the compact containment condition. 


\item[(c)] Since the noise is modelled by a generalised process, not attaining values in the underlying Hilbert space, one cannot directly apply  Skorokhod's theorem to conclude existence of strong solutions on another probability space from existence of weak solutions. We circumvent this problem by decoding the cylindrical noise as a random variable in the space of c\`adl\`ag functions with values in $\R^\infty$, and re-building the noise after the application of Skorokhod's theorem.

\item[(d)] The setting of the equation under consideration requires  to apply the Yamada--Watanabe 
result in a rather generalised and abstract form due to Kurtz \cite{Kurtz_Yamada}. For this  purpose, we have to interpret the evolution equation as a `stochastic model' in the sense of Kurtz  with  the initial condition and the noise as the input variables and the solution as the output variable. In this abstract setting, the equation itself is considered as a measurable constraint on the input-output space.
\end{itemize}

We describe the content of the paper. Some preliminaries on $\alpha$-stable cylindrical L\'evy processes are presented in Section 2. In Section \ref{sec_tail_estimate}, we obtain tail inequalities for the stochastic integrals with respect to stable cylindrical L\'evy process $L$ on a Hilbert space. 
Section \ref{sec_auxiliary} contains auxiliary analytical lemmas regarding tightness and the Skorokhod topology used in the proof of existence and uniqueness.
Section \ref{sec_SPDE} comprises of all the steps mentioned above to conclude the existence of a mild solution of \eqref{intro_stochastic_evolution_equation}. 
It contains Theorem \ref{th_existence_weak_sol_main_result} on the existence and uniqueness of solution to the SPDE \eqref{intro_stochastic_evolution_equation}, which is the main result of the paper.

\section{Preliminaries}
\label{sec_preliminaries}

Let $H$ be a separable Hilbert space. 
We recall some facts about stable measures from \cite{Linde}. A probability measure $\mu$ on the $\sigma$-algebra $\mathcal{B}(H)$ of Borel subsets of $H$ is called stable if for every $n\in \N$ there exists $\gamma_n>0$ and $x_n \in H$ such that the characteristic function satisfies 
\begin{equation}
\label{stable_measure_definition}
\phi_\mu (h) ^n = \phi_\mu(\gamma_n h) e^{i\langle x_n, h \rangle} \qquad \text{for all } h \in H.
\end{equation} 
Each stable measure is infinitely divisible and its L\'evy measure $\nu$ can be written as 
$$\nu(B) = c_\alpha^{-1} \int_0^\infty \int_{S_H} \1_B(tx) \, \sigma(\ud x) t^{-1-\alpha} \dt, \qquad B\in \mathcal{B}(H),$$
where $S_H$ is the sphere of radius $1$ in $H$, $\sigma$ is a finite measure on $S_H$ and $c_\alpha$ is defined as
\begin{equation*}
c_\alpha := 
\begin{cases}
-\alpha \cos(\frac{\alpha \pi}2) \Gamma(-\alpha), & \text{for } \alpha \neq 1, \\
\frac{\pi}2, & \text{for } \alpha = 1.
\end{cases}
\end{equation*}
The measure $\sigma$ is called the spectral measure of $\mu$ and it can be recovered from $\nu$ as
\begin{equation}
\label{formula_for_spectral_measure}
\sigma(B) = \alpha c_\alpha \nu\left( \left\{ x \in H : \norm{x}>1, \frac{x}{\norm{x}} \in B \right\} \right), \qquad B\in \mathcal{B}(S_H).
\end{equation}
By \cite[Prop.\ 7.5.4(iv)]{Linde} for every Hilbert space $H$ there exists a constant $c>0$ such that for every stable measure $\mu$ on $\mathcal{B}(H)$
\begin{equation}
\label{sup_bounded_by_lim}
\sup_{r>0} r^\alpha \mu(\norm{x}>r) \leq c \lim_{r \to \infty} r^\alpha \mu(\norm{x}>r).
\end{equation}

We recall definitions and properties of cylindrical random variables and stable cylindrical L\'evy processes. 
We denote the space of equivalence classes of real-valued random variables on $\Omega$ by $L^0(\Omega,\mathcal{F},P)$. It is a metric space under the Ky Fan metric, which induces the convergence in probability.
Let $U$ be a separable Hilbert space, whose dual space is identified with $U$. A cylindrical random variable is a linear and continuous mapping $X\colon U \to L^0(\Omega,\mathcal{F},P)$. Its characteristic function is a function $\phi_X \colon U \to \C$ defined by $\phi_X(u) = \E \left[ e^{iXu} \right]$.

Let $(\mathcal{F}_t)$ be a filtration on $\Omega$. A cylindrical L\'evy process is a family of cylindrical random variables $\big( L(t) : t \geq 0 \big)$ with $L(t)\colon U \to L^0(\Omega,\mathcal{F},P)$ such that all projections $\big( (L(t)u_1,\ldots, L(t)u_n) : t\geq 0 \big)$
for $n\in \N$ and $u_1,\ldots, u_n \in U$, define L\'evy processes in $\R^n$ with respect to $(\mathcal{F}_t)$.
Let $\alpha \in (0,2]$. A canonical cylindrical $\alpha$-stable L\'evy process is a cylindrical L\'evy process whose characteristic function satisfies $\phi_{L(t)}(u) = \e^{-t \norm{u}^\alpha}$ for $u \in U$; see \cite{Applebaum_Riedle,Riedle_stable}.

Cylindrical processes are intertwined with the concept of cylindrical measures, which we now recall.
For $\Delta \subset U$ let $\mathcal{Z}(U,\Delta)$ be the collection of the cylindrical sets
$$C(u_1,\ldots,u_n;B):= \{ u \in U : (\langle u,u_1 \rangle, \ldots, \langle u,u_n\rangle) \in B \},$$
where $u_1,\ldots,u_n \in \Delta$ and $B\in \mathcal{B}(\R^n)$. The family of all cylindrical subsets of $U$ is $\mathcal{Z}(U) := \mathcal{Z}(U,U)$.
A mapping $\mu\colon  \mathcal{Z}(U)\to \R_+$ such that its restriction to $\mathcal{Z}(U,\Delta)$ is a measure for any finite set $\Delta$ is called a cylindrical measure. 

For the cylindrical distribution of $L(1)$ defined by
$$\mu(C(u_1,\ldots,u_n;B)) = P( (L(1)u_1,\ldots, L(1)u_n) \in B)
\qquad\text{for }B\in\mathcal{B}(\R^n), $$ 
one can derive the counterpart of the L\'evy-Khintchine formula, see \cite{Riedle_infinitely,Riedle_stable}. It is characterised by the so-called cylindrical L\'evy measure
$\nu\colon \mathcal{Z}(U)\to [0,\infty]$ which satisfies
\begin{equation}
\label{cyl_Levy_measure_of_stable}
\nu \circ \pi_{e_1,\ldots,e_n}^{-1}(B) = \frac{\alpha}{c_\alpha} \int_{S_{\R^n}} \int_0^\infty \1_B(rx) \frac1{r^{1+\alpha}} \dr \, \lambda_n(\ud x),\qquad B \in\mathcal{B}(\R^n),
\end{equation}
where $(e_k)$ is an orthonormal basis of $U$, $\pi_{e_1,\ldots,e_n}\colon  U \to \R^n$ is given by $\pi_{e_1,\ldots,e_n}(u) = (\langle u,e_1\rangle, \ldots ,\langle u,e_n\rangle)$ and the measure $\lambda_n$ on $S_{\R^n}$ is uniform with the total mass
\begin{equation}
\lambda_n(S_{\R^n}) = \frac{\Gamma(\frac12) \Gamma(\frac{n+\alpha}2)}{\Gamma(\frac{n}2) \Gamma(\frac{1+\alpha}2)}, 
\end{equation}
where $\Gamma$ denotes the Gamma function. 

Stochastic integration with respect to cylindrical L\'evy processes was introduced in \cite{Jakubowski_Riedle}. We denote the space of the Hilbert-Schmidt operators from $U$ to $H$ by $L_{\rm HS}(U,H)$.
In \cite{Jakubowski_Riedle} the class $\Lambda:=\Lambda(U,H)$ of admissible integrands  consists of $L_{\rm HS}(U,H)$-valued adapted processes on $[0,T]$ with c\`agl\`ad  paths.  
The integral process is a c\`adl\`ag $H$-valued semimartingale and the following continuity property holds: if $\Psi_n \to \Psi$ in probability in the Skorokhod space of c\`agl\`ad functions on $[0,T]$ with values in $L_{\rm HS}(U,H)$, then for all $t\in [0,T]$
$$\int_0^t \Psi_n(s) \dL(s) \to \int_0^t \Psi(s) \dL(s)$$
in probability for all $t\in [0,T]$.
The Hilbert-Schmidt operators play a crucial role here, since they map cylindrical random variables into classical ones: if $\psi \in L_{\rm HS}(U,H)$, then for any $t\geq 0$ there exists an $\mathcal{F}_t$-measurable random variable $\psi(L(t)) \colon \Omega \to H$ such that for all $h\in H$ one has $L(t)(\psi^*h) = \langle \psi(L(t)),h \rangle$.

We recall some facts about the fractional powers of positive operators, which can be found e.g.\ in \cite{Lunardi,Pazy}. Suppose that $A$ is a generator of a strongly continuous semigroup $(S(t) : t \geq 0)$ on $H$ and assume that $\{0\} \cup \{ \lambda \in {\mathbb C} : 0<\omega<\abs{\arg \lambda} \leq \pi\}$ is contained in the resolvent set $\rho(-A)$ for some $\omega<\frac{\pi}{2}$.
One defines for $\delta>0$
\begin{equation}
(-A)^{-\delta} 
= \frac{1}{2\pi i} \int_\gamma \lambda^{-\delta} (\lambda+A)^{-1} \, \ud \lambda,
\end{equation}
where $\gamma$ is a curve running from $\infty e^{i\theta}$ to $\infty e^{-i\theta}$ with $\omega < \theta < \pi$ avoiding the negative axis and the origin. It can be shown that for $\delta>0$ the operator $(-A)^{-\delta}$ is bounded and one-to-one.
We define for $\delta>0$
$$(-A)^{\delta} \colon  D((-A)^\delta):= \mathcal{R}((-A)^{-\delta}) \to H, \qquad (-A)^{\delta} := \left( (-A)^{-\delta} \right)^{-1}.$$
The operators $(-A)^\delta$ are closed. 
We equip the space $D((-A)^\delta)$ with the norm defined as $\norm{x}_{D((-A)^\delta)} := \bignorm{(-A)^\delta x}$ for $x \in D((-A)^\delta)$.
This norm is equivalent to the graph norm.
The embedding $D((-A)^\delta) \subset H$ is continuous.
With this notation in mind, \cite[Th.\ 2.6.13(d)]{Pazy} states that for all $\delta \in (0,1]$ one has for some $C>0$
\begin{equation}
\label{norm_continuity_of_semigroup}
\norm{S(t)-\Id}_{L(D((-A)^\delta),H)} \leq C t^\delta,
\end{equation}
where $\Id$ denotes the identity operator on $H$.

\section{Tail and moment estimates}
\label{sec_tail_estimate}

In this section we assume that $\alpha \in (0,2)$. We prove a tail estimate for a Radonified stable cylindrical measure and then we generalise the tail estimate of the stochastic integral to the case of an integral with respect to a stable cylindrical process. Finally, we derive a moment inequality for that stochastic integral.

\begin{Lemma}
\label{lem_tail_estimate_radonified}
For any canonical $\alpha$-stable cylindrical L\'evy process $L$ on $U$ and $\psi \in L_{\rm HS}(U,H)$ we have
\begin{equation}
\label{tail_of_radonified_Levy_measure}
\sup_{r>0} r^\alpha P( \norm{\psi(L(t))} > r)
\leq  c t\left( \nu \circ \psi^{-1}\right) (\bar B_H^c) 
\leq c_{1,\alpha} t \HSnorm{\psi}^\alpha,
\end{equation}
where $\bar B_H^c$ is the complement of the closed unit ball, $c_{1,\alpha} = \tfrac{ c \Gamma(\frac12)}{c_\alpha \Gamma(\frac{1+\alpha}2)}$
and $c$ is the constant (depending on $\alpha$) appearing in \eqref{sup_bounded_by_lim}.
\end{Lemma}

\begin{proof}
Note that the characteristic function of $\psi(L(t))$ is given by 
$$\phi_{\psi(L(t))}(h) 
= \E \left[ e^{i L(t)(\psi^* h)} \right] 
= e^{-t \norm{\psi^*h}^\alpha} \qquad \text{for } h \in H,$$
and thus it is a stable random variable, cf.\ \eqref{stable_measure_definition}.
It follows by \eqref{sup_bounded_by_lim} that
\begin{equation}
\label{sup_bounded_by_lim_radonified}
\sup_{r>0} r^\alpha P(\norm{\psi(L(t))}>r) 
\leq c \lim_{r\to \infty} r^\alpha P(\norm{\psi(L(t))}>r).
\end{equation}
The L\'evy measure of the infinitely divisible random variable $\psi(L(t))$ is $t\left(\nu \circ \psi^{-1}\right)$. 
Note that since $\psi$ is Hilbert-Schmidt, the measure $\nu \circ \psi^{-1}$ is in fact a genuine L\'evy measure.

Let $\sigma$ denote the spectral measure of the stable random variable $\psi(L(1))$. 
Combining \cite[Cor.\ 6.7.3]{Linde} and formula \eqref{formula_for_spectral_measure} we get
\begin{equation}
\label{lim_equals_Levy_measure}
\lim_{r\to \infty} r^\alpha P(\norm{\psi(L(t)))}>r)
= t \frac{\sigma(S_H)}{\alpha c_\alpha}
= t\left( \nu \circ \psi^{-1} \right)(\bar B_H^c).
\end{equation}
Now the first inequality in \eqref{tail_of_radonified_Levy_measure} follows from \eqref{sup_bounded_by_lim_radonified} and \eqref{lim_equals_Levy_measure}. 

For establishing  the second inequality in \eqref{tail_of_radonified_Levy_measure}, note that 
the operator $\psi$ has the decomposition
$$\psi = \sum_{n=1}^\infty \gamma_n (e_n \otimes f_n),$$
where $(e_n)$ is an orthonormal system in $U$, $(f_n)$ is an orthonormal system in $H$ and $(\gamma_n) \subset \R$, see \cite[Th.\ 4.1]{Diestel}.
Let $P_n\colon H \to H$ be the projection onto $\Span(f_1, \ldots, f_n)$. Since $P_n\psi \to \psi$ in $L_{\rm HS}(U,H)$, it follows by the continuity of the integral in \cite[Th.\ 5.1]{Jakubowski_Riedle} that $P_n\psi(L(1)) = \int_0^1 P_n\psi \dL(s)$ converges in probability to $\psi(L(1)) = \int_0^1 \psi \dL(s)$ as $n\to \infty$. Proposition~6.6.5 in \cite{Linde} implies that the spectral measures $\sigma_n$ of $(P_n\psi)(L(1))$ converge weakly to the spectral measure $\sigma$ of $\psi(L(1))$. 
By \cite[Prop.\ 1.3.2.]{Linde} the total mass of $\sigma_n$ converges, i.e.\ $\sigma_n(S_H) \to \sigma(S_H)$ as $n \to \infty$.
It follows from \eqref{formula_for_spectral_measure} that $\sigma(S_H) = \alpha c_\alpha \left( \nu \circ \psi^{-1} \right) (\bar B_H^c)$ and $\sigma_n(S_H) = \alpha c_\alpha \left( \nu \circ \psi^{-1} \circ P_n^{-1} \right) (\bar B_H^c)$ and thus
\begin{align}
\label{eq.convergence-total-sigma}
\left( \nu \circ \psi^{-1} \right)(\bar B_H^c) = \lim_{n\to \infty} \left( \nu \circ \psi^{-1} \circ P_n^{-1} \right) (\bar B_H^c).
\end{align}
We calculate by \eqref{cyl_Levy_measure_of_stable}
\begin{align*}
\left( \nu \circ \psi^{-1} \circ P_n^{-1} \right)(\bar B_H^c)
&= \nu \circ \pi_{e_1,\ldots,e_n}^{-1}\bigg( \bigg\{ x \in \R^n : \sum_{j=1}^n \gamma_j^2 x_j^2 >1 \bigg\} \bigg) \\
&= \frac{\alpha}{c_\alpha} \int_{S_{\R^n}} \int_0^\infty \1_{\left\{y \in \R^n : \sum_{j=1}^n  \gamma_j^2 y_j^2>1 \right\}}(rx) \frac1{r^{1+\alpha}} \dr \, \lambda_n(\ud x) \\
&= \frac{1}{c_\alpha} \int_{S_{\R^n}} \bigg(\sum_{j=1}^n  \gamma_j^2 x_j^2\bigg)^{\alpha/2} \, \lambda_n(\ud x).
\end{align*}
Denote the probability measure $\lambda_n^{(1)} := \tfrac{1}{\lambda_n(S_{\R^n})} \lambda_n$.
Jensen's inequality implies that
\begin{align*}
\left(\nu \circ \psi^{-1} \circ P_n^{-1}\right)(\bar B_H^c)
&\leq \frac{\lambda_n(S_{\R^n})}{c_\alpha} \Bigg( \sum_{j=1}^n \gamma_j^2 \int_{S_{\R^n}}  x_j^2  \, \lambda_n^{(1)}(\ud x)  \Bigg)^{\alpha/2}= \frac{\lambda_n(S_{\R^n})}{c_\alpha n^{\alpha/2}} \Bigg( \sum_{j=1}^n \gamma_j^2 \Bigg)^{\alpha/2},
\end{align*} 
because $\int_{S_{\R^n}}  x_j^2 \, \lambda_n^{(1)}(\ud x) = \frac1{n}$ for all $j=1,\ldots,n$.
Recalling  $\frac{\Gamma(x+\beta)}{\Gamma(x)x^\beta} \to 1$ as $x \to \infty$ and taking  the limit
as $n\to\infty$ in the inequality above completes the proof due to \eqref{eq.convergence-total-sigma}.
\end{proof}

\begin{Theorem}
\label{th_tail_estimate}
Any stochastic process $\Psi$ in the space $\Lambda$ of admissible integrands satisfies
\begin{equation}
\label{integral_tail_estimate}
\sup_{r>0} r^\alpha P\left( \sup_{t \in [0,T]} \norm{\int_0^t \Psi(s) \dL(s)} >r \right) 
\leq c_{2,\alpha} \E \left[ \int_0^T \HSnorm{\Psi(s)}^\alpha \ds \right]
\end{equation}
with $c_{2,\alpha} = c_{1,\alpha} \frac{4-\alpha}{2-\alpha}$.
\end{Theorem}

\begin{proof}
In the proof we follow \cite[Th.\ 4.3]{Balan} and \cite[Lem.\ 3.3]{Gine_Marcus}.
Fix a simple process $\Psi$, which takes only finitely many values and is based on a partition $0=s_1<t_2<\ldots<s_K=T$, and denote for $t \in [0,T]$
$$I(t) := \int_0^t \Psi(s) \dL(s).$$
Let $\{t_1,\ldots,t_N\}$ be a partition of $[0,T]$ containing $\{s_1,\ldots,s_K\}$.
One can write $\Psi$ as
$$\Psi = \Psi_0 \1_{\{0\}} + \sum_{i=1}^{N-1} \Psi_i \1_{(t_i,t_{i+1}]}, \qquad \Psi_i = \sum_{j=1}^{m_i} \1_{A_{i,j}} \psi_{i,j},$$
where for each $i=1,\ldots ,N-1$ the sets $A_{i,1},\ldots, A_{i,m_i} \in \mathcal{F}_{t_i}$ form a partition of $\Omega$.
Then
$$I(T) 
= \sum_{i=1}^{N-1} \Psi_i(L(t_{i+1})-L(t_i))
= \sum_{i=1}^{N-1} \sum_{j=1}^{m_i} \psi_{i,j}(L(t_{i+1})-L(t_i))\1_{A_{i,j}}.$$
We have
\begin{align}
\label{split_the_probability_into_two}
&P\left( \max_{i=1,\ldots,N} \norm{I(t_i)} >r \right) \notag \\
&\quad\quad\leq \sum_{i=1}^{N-1} P\left( \norm{\Psi_i(L(t_{i+1})-L(t_i))} > r \right) \notag \\
&\quad\quad\quad + P \left( \max_{k=1,\ldots,N-1} \norm{\sum_{i=1}^{k} \Psi_i(L(t_{i+1})-L(t_i)) \1_{\{\smallnorm{\Psi_i(L(t_{i+1})-L(t_i))} \leq r\}}} >r \right)  \notag \\
&\quad\quad =: p_1 + p_2 .
\end{align}
We estimate the terms $p_1$ and $p_2$ separately.
Since for fixed $i$ the sets $A_{i,j}$ are independent from $\psi_{i,j}(L(t_{i+1})-L(t_i))$ for all $j=1,\dots, m_i$, we obtain
\begin{align}
\label{calculation_by_independence}
P\left(\norm{\Psi_i(L(t_{i+1})-L(t_i))}>r\right) 
&= P\Bigg( \Bigg\Vert \sum_{j=1}^{m_i} \1_{A_{i,j}} \psi_{i,j}(L(t_{i+1})-L(t_i)) \Bigg\Vert > r \Bigg) \notag \\
&= \sum_{j=1}^{m_i} P(A_{i,j}) P\left(\norm{\psi_{i,j}(L(t_{i+1})-L(t_i)}>r\right).
\end{align}
From Lemma \ref{lem_tail_estimate_radonified} we obtain
\begin{align*}
P\left(\norm{\Psi_i(L(t_{i+1})-L(t_i))}>r\right) 
&\leq c_{1,\alpha} r^{-\alpha} (t_{i+1}-t_i) \sum_{j=1}^{m_i} P(A_{i,j}) \HSnorm{\psi_{i,j}}^\alpha \\
&= c_{1,\alpha} r^{-\alpha} (t_{i+1}-t_i) \E \left[ \HSnorm{\Psi_i}^\alpha \right].
\end{align*}
Consequently, we obtain that
$$p_1 \leq c_{1,\alpha} r^{-\alpha} \E \left[ \int_0^T \HSnorm{\Psi(s)}^\alpha \ds \right].$$
We estimate the second term on the right-hand side of \eqref{split_the_probability_into_two}.
It follows from the symmetry of the cylindrical distribution of $L(t)-L(s)$ that the Radonified random variables $\psi(L(t)-L(s))$ are symmetric as well and thus the following discrete process is a martingale
$$\left( \sum_{i=1}^k \Psi_i(L(t_{i+1})-L(t_i)) \1_{\{\smallnorm{\Psi_i(L(t_{i+1})-L(t_i))} \leq r\}} : k=1,
\ldots, N-1 \right).$$
Doob's inequality implies 
\begin{align}
\label{estimate_of_p_2}
p_2
&\leq r^{-2}  \E \left[ \norm{\sum_{i=1}^{N-1} \Psi_i(L(t_{i+1})-L(t_i))}^2 \1_{\{\smallnorm{\Psi_i(L(t_{i+1})-L(t_i))} \leq r\}} \right] \notag \\
&= r^{-2}  \sum_{i=1}^{N-1} \E \left[ \norm{\Psi_i(L(t_{i+1})-L(t_i))}^2 \1_{\{\smallnorm{\Psi_i(L(t_{i+1})-L(t_i))} \leq r\}} \right],
\end{align}
where the last equality follows from the orthogonality in $L^2(\Omega,\mathcal{F},P)$ of the summands
for different indices.  
We write similarly to \eqref{calculation_by_independence}
\begin{multline}
\label{estimate_of_single_term}
\E \left[ \norm{\Psi_i(L(t_{i+1})-L(t_i))}^2 \1_{\{\smallnorm{\Psi_i(L(t_{i+1})-L(t_i))} \leq r\}} \right] \\
= \sum_{j=1}^{m_i} P(A_{i,j}) \E \left[ \norm{\psi_{i,j}(L(t_{i+1})-L(t_i))}^2 \1_{\{\smallnorm{\psi_{i,j}(L(t_{i+1})-L(t_i))} \leq r\}}  \right].
\end{multline}
Fubini's theorem and Lemma \ref{lem_tail_estimate_radonified} imply that
\begin{align}\label{estimate_of_expectation}
&\E \left[ \norm{\psi_{i,j}(L(t_{i+1})-L(t_i))}^2 \1_{\{\smallnorm{\psi_{i,j}(L(t_{i+1})-L(t_i))} \leq r\}}  \right]\notag \\
&\qquad\qquad = 2 \int_0^r t P\left(t<\norm{\psi_{i,j}(L(t_{i+1})-L(t_i))} \leq r\right) \dt \notag\\
&\qquad\qquad  \leq 2 \int_0^r t P\left(t< \norm{\psi_{i,j}(L(t_{i+1})-L(t_i))}\right) \dt\notag\\
&\qquad \qquad \leq 2c_{1,\alpha} (t_{i+1}-t_i) \HSnorm{\psi_{i,j}}^\alpha \int_0^r t^{1-\alpha} \dt \notag\notag \\
&\qquad \qquad = \frac{2c_{1,\alpha}}{2-\alpha} (t_{i+1}-t_i) \HSnorm{\psi_{i,j}}^\alpha r^{2-\alpha}.
\end{align}
Combining \eqref{estimate_of_p_2}, \eqref{estimate_of_single_term} and \eqref{estimate_of_expectation} we get
$$p_2 
\! \leq \! \frac{2c_{1,\alpha}}{2-\alpha} \sum_{i=1}^{N-1} \sum_{j=1}^{m_i} (t_{i+1}-t_i) P(A_{i,j}) \HSnorm{\psi_{i,j}}^\alpha r^{-\alpha} 
= \frac{2c_{1,\alpha}}{2-\alpha} r^{-\alpha} \E \left[ \int_0^T \! \!\HSnorm{\Psi(s)}^\alpha \ud s \right]\!\!.$$
We have shown that 
\begin{equation}
\label{inequality_for_any_partition}
P\left( \max_{i=1,\ldots,n} \norm{I(t_i)} >r \right)
\leq c_{2,\alpha} r^{-\alpha} \E \left[ \int_0^T \HSnorm{\Psi(s)}^\alpha \ds \right]
\end{equation}
for any partition containing $\{s_1,\ldots, s_K\}$.

The process $\left( I(t) : t \in [0,T] \right)$ has c\`adl\`ag paths and thus for any sequence $(\zeta_n)$ of finite subsets of $[0,T]$, which increases to a countable and dense subset $\zeta\subset [0,T]$ we have 
$$\lim\limits_{n\to \infty} \sup\limits_{t\in \zeta_n} \norm{I(t)}
= \sup\limits_{t\in \zeta} \norm{I(t)}
= \sup\limits_{t\in [0,T]} \norm{I(t)}.$$
It follows that there exists a sequence of partitions $0=t_1^n<\ldots<t_{k_n}^n=T$ such that
$$P\bigg( \sup_{t\in [0,T]} \norm{I(t)} >r \bigg)
= \lim_{n\to \infty} P\left( \max_{i=1,\ldots,k_n} \norm{I(t_i^n)} >r \right).$$
This combined with \eqref{inequality_for_any_partition} finishes the proof for simple prcesses $\Psi$.

Now, consider the case of a general $\Psi$ in the space $\Lambda$ of admissible integrands. If the right-hand side of \eqref{integral_tail_estimate} is infinite, then the inequality holds trivially. Otherwise, it follows as in Proposition 4.22(ii) and Lemma 1.3 in \cite{Da_Prato_Zabczyk} that there exists a sequence of c\`agl\`ad simple processes $\Psi_n$, such that $d_\alpha(\Psi_n, \Psi) \to 0$ as $n\to \infty$, where
$$d_\alpha(\Psi, \Phi) := 
\begin{cases}
\E \left[ \int_0^T \HSnorm{\Psi(s)-\Phi(s)}^\alpha \ds \right], & \text{if } \alpha< 1, \\
\left(\E \left[ \int_0^T \HSnorm{\Psi(s)-\Phi(s)}^\alpha \ds \right]\right)^{1/\alpha}, & \text{if } \alpha\geq 1.
\end{cases}$$
The estimate \eqref{integral_tail_estimate} for simple processes implies that $\int \Psi_n(s)\dL(s)$ is a Cauchy sequence in the topology of uniform convergence in probability and thus converges in this topology to $\int \Psi(s)\dL(s)$.
We have by the Portmanteau theorem
\begin{align*}
P\left( \sup_{t\in [0,T]} \norm{\int_0^t \Psi(s) \dL(s)} > r \right)
&\leq \liminf_{n\to \infty} P\left( \sup_{t \in [0,T]} \norm{\int_0^t \Psi_n(s) \dL(s)} > r \right) \\
&\leq \liminf_{n\to \infty} c_{2,\alpha} r^{-\alpha} \E \left[ \int_0^T \HSnorm{\Psi_n(s)}^\alpha \ds \right] \\
&= c_{2,\alpha} r^{-\alpha} \E \left[ \int_0^T \HSnorm{\Psi(s)}^\alpha \ds \right].
\end{align*}
Moving $r^{-\alpha}$ to the left-hand side and taking supremum over $r>0$ we get the claim.
\end{proof}


We generalise the moment inequality for the stochastic integrals, which was proved for  square-integrable martingales in \cite[Th.\ 3.41, Th.\ 9.24]{Peszat_Zabczyk} and for vector-valued stable processes in \cite{Rosinski_Woyczynski_moment}.

\begin{Corollary}
\label{pro_moment_inequality_2}
Any stochastic process $\Psi$ in the space $\Lambda$ of admissible integrands satisfies
$$\E \left[ \sup_{t \in [0,T]} \norm{\int_0^t \Psi(s) \dL(s)}^p \right]
\leq C_{\alpha,p} \left( \E \left[ \int_0^T \HSnorm{\Psi(s)}^\alpha \ds \right] \right)^{p/\alpha}$$
for $p<\alpha$  where $ C_{\alpha,p}:=\tfrac{c_{2,\alpha}^{p/\alpha}\alpha}{\alpha-p}$.
\end{Corollary}

\begin{proof}
Let $X:= \sup\limits_{t\in [0,T]} \norm{\displaystyle\int_0^t \Psi(s) \dL(s)}$ and  $\xi := c_{2,\alpha} \E \left[\displaystyle\int_0^T \HSnorm{\Psi(s)}^\alpha \ds\right]$. Theorem~\ref{th_tail_estimate} implies that for $r>0$ we have $P(X>r) \leq 1 \wedge (r^{-\alpha} \xi)$.
Therefore we obtain
\begin{align*}
\E \left[ \sup_{t \in [0,T]} \norm{\int_0^t \Psi(s) \dL(s)}^p \right]
&= p\int_0^\infty r^{p-1} P(X>r) \dr \\
&\leq p \int_0^\infty r^{p-1} \left( 1 \wedge (r^{-\alpha} \xi ) \right) \ud r\\
&= \left( 1 + \frac{p}{\alpha-p} \right) \xi^{p/\alpha} \\
&= \frac{c_{2,\alpha}^{p/\alpha} \alpha}{\alpha-p} \left( \E \left[ \int_0^T \HSnorm{\Psi(s)}^\alpha \ds \right] \right)^{p/\alpha}.
\qedhere
\end{align*}
\end{proof} 

\section{Auxiliary analytical lemmas}
\label{sec_auxiliary}

\subsection{Convergence of stochastic integrals}

We formulate the following convergence result for the stochastic integral in the form  needed later. 
Generalising this result to the usual convergence of stochastic integrals under the UT-condition, which does not have an obvious analogue in the cylindrical setting here, will be investigated in future work.

\begin{Lemma}
\label{lem_general_integrand}
Suppose that $L_k$ for $k \in \N$ and $L$ are canonical $\alpha$-stable processes for $\alpha\in (1,2)$  such that $L_k(t)u \to L(t)u$ in probability for all $u \in U$ and $t \in [0,T]$. 
If a sequence  $\big(\Psi_k)_{k\in\N}\subseteq  L^\alpha([0,T] \times \Omega;L_{\rm HS}(U,H))$ 
converges to some $ \Psi$ in $ L^\alpha([0,T] \times \Omega;L_{\rm HS}(U,H))$ then it follows that
\begin{equation*}
\int_0^t \Psi_k(s) \dL_k(s) \to \int_0^t \Psi(s) \dL(s)
\end{equation*}
in probability for all $t \in [0,T]$.
\end{Lemma}

\begin{proof}
Step 1. For a deterministic function  $\phi \in L_{\rm HS}(U,H)$ and $t\in [0,T]$, it follows that the set $\{P_{\phi(L_k(t))}:\, k\in\N\}$ of probability distributions $P_{\phi(L_k(t))}$ of the random variable $\phi(L_k(t))$ is tight, since each $L_k$ has the same cylindrical distribution. Since we have for each $h\in H$ that
$$\lim_{k\to\infty}\langle \phi(L_k(t)),h \rangle
=\lim_{k\to\infty} L_k(t)(\phi^* h)
= L(t)(\phi^*h)
= \langle \phi(L(t)),h \rangle \qquad\text{in }L^0(\Omega;\R),  $$
Lemma \cite[Lem.\ 2.4]{Jakubowski_1988} implies that 
 $\phi(L_k(t)) \to \phi(L(t))$ in $L^0(\Omega;H)$.

Step 2. We establish for each $\Psi\in  L^\alpha([0,T] \times \Omega;L_{\rm HS}(U,H))$ and $t\in [0,T]$ that
\begin{equation*}
\lim_{k\to\infty} \norm{\int_0^t \Psi(s) \dL_k(s) - \int_0^t \Psi(s) \dL(s)} = 0
 \qquad\text{in probability. }
\end{equation*}
For this purpose, suppose first that $\Phi\colon [0,T] \times \Omega \to L_{\rm HS}(U,H)$ is a simple integrand of the form
\begin{equation*}
\Phi(s) = \sum_{j=1}^{N-1} \Phi_j \1_{(t_j,t_{j+1}]}(s) \qquad \text{for }
\Phi_j = \sum_{l=1}^{m_j} \1_{A_{j,l}} \phi_{j,l}, 
\end{equation*}
where $\phi_{j,l}$ are deterministic functions in $L_{\rm HS}(U,H)$, the sets $A_{j,l}$ are in 
$\mathcal{F}_{t_j}$ and $0=t_1<\ldots< t_N=t$. Step 1 implies that 
\begin{align*}
\lim_{k\to\infty}\int_0^t \Phi(s) \dL_k(s) 
&= \lim_{k\to\infty} \sum_{j=1}^{N-1} \sum_{l=1}^{m_j} \1_{A_{j,l}} \phi_{j,l}(L_k(t_{j+1}) - L_k(t_j))\\
&= \sum_{j=1}^{N-1} \sum_{l=1}^{m_j} \1_{A_{j,l}} \phi_{j,l}(L(t_{j+1}) - L(t_j))
=\int_0^t \Phi(s) \dL(s)
\end{align*}
in probability.

For an arbitrary integrand $\Psi\in L^\alpha([0,T] \times \Omega;L_{\rm HS}(U,H))$  there exists a sequence $(\Phi_n)_{n\in\N}$ of simple integrands of the above form converging to $\Psi$ in $L^\alpha([0,T] \times \Omega; L_{\rm HS}(U,H))$. Thus, for each $\epsilon>0$  there exists $n_0\in\N$ such  that 
\begin{equation*}
c_{2,\alpha} \left(\tfrac{\epsilon}{3} \right)^{-\alpha} \E \left[ \int_0^T \norm{\Phi_n(s)-\Psi(s)}_{L_{\rm HS}(U,H)}^\alpha \ds \right] 
\leq \frac{\epsilon}{3} \qquad\text{for all }n\ge n_0.
\end{equation*}
Consequently, Theorem \ref{th_tail_estimate} implies  for each $n\ge n_0$ and for all $k \in \N$ that
\begin{align*}
P\left(\norm{\int_0^t \Psi(s) \dL_k(s) - \int_0^t \Phi_{n}(s) \dL_k(s)} \geq \frac{\epsilon}{3}\right) &\leq \frac{\epsilon}{3},\\
P\left(\norm{\int_0^t \Phi_{n}(s) \dL(s) - \int_0^t \Psi(s) \dL(s)} \geq \frac{\epsilon}{3}\right) &\leq \frac{\epsilon}{3}.
\end{align*}
As $\Phi_{n_0}$ is simple, the argument above guarantees that there exists $k_0=k_0(n_0)\in\N$ such that for all $k \geq k_0$ we have 
\begin{equation*}
P\left(\norm{\int_0^t \Phi_{n_0}(s) \dL_k(s) - \int_0^t \Phi_{n_0}(s) \dL(s)} \geq \frac{\epsilon}{3}\right) \leq \frac{\epsilon}{3}.
\end{equation*}
In summary, it follows for each $k \geq k_0$ that 
\begin{align*}
&P\left(\norm{\int_0^t \Psi(s) \dL_k(s) - \int_0^t \Psi(s) \dL(s)} \geq \epsilon\right) \\
&\qquad \leq P\left(\norm{\int_0^t \Psi(s) \dL_k(s) - \int_0^t \Phi_{n_0}(s) \dL_k(s)} \geq \frac{\epsilon}{3}\right) \\
&\qquad \qquad + P\left(\norm{\int_0^t \Phi_{n_0}(s) \dL_k(s) - \int_0^t \Phi_{n_0}(s) \dL(s)} \geq \frac{\epsilon}{3}\right) \\
&\qquad \qquad + P\left(\norm{\int_0^t \Phi_{n_0}(s) \dL(s) - \int_0^t \Psi(s) \dL(s)} \geq \frac{\epsilon}{3}\right) 
\leq \frac{\epsilon}{3} + \frac{\epsilon}{3} + \frac{\epsilon}{3} = \epsilon,
\end{align*}
which completes Step 2. 

Step 3. Let $\Psi_k$ be as in the statement of the lemma, that is converge to $\Psi$ in $L^\alpha([0,T]\times \Omega;L_{\rm HS}(U,H))$. We have
\begin{multline}
\label{by_triangle_inequality}
\norm{\int_0^t \Psi_k(s) \dL_k(s)- \int_0^t \Psi(s) \dL(s)} \\
\leq \norm{\int_0^t \Psi_k(s) \dL_k(s) - \int_0^t \Psi(s) \dL_k(s)}  + \norm{\int_0^t \Psi(s) \dL_k(s) - \int_0^t \Psi(s) \dL(s)}
\end{multline}
By applying Step 2 we conclude that the second term on the right-hand side converges to $0$ in probability.
Corollary \ref{pro_moment_inequality_2} implies for $p \in (1,\alpha)$ that 
\begin{align*}
\E \left[ \sup_{t \in [0,T]} \norm{\int_0^t \Psi_k(s) \dL_k(s) - \int_0^t \Psi(s) \dL_k(s)}^p \right]
&\leq C_{\alpha,p} \left( \E \left[ \int_0^T \norm{\Psi_k(s) - \Psi(s)}^\alpha \ds \right] \right)^{p/\alpha},
\end{align*}
which by our assumption establishes for each $t\in [0,T]$ that the first term on the right-hand side of \eqref{by_triangle_inequality} converges to $0$ in probability. This finished the proof.
\end{proof}

\subsection{Distribution of the stochastic integral}

In the next lemma, we show that the joint distribution of the integrand and integrator uniquely determines the distribution of the stochastic integral. A similar result for integrals with respect to Poisson random measures was derived in \cite{Brzezniak_Hausenblas_uniqueness}. Here, as the integrator is a process of classical random variables but the integrator of cylindrical random variables, we write explicitly  the condition on the joint distribution. We consider this condition in a more formal context in the following Remark \ref{re.classical-cylindrical}. We denote equality in distribution by $\stackrel{\mathcal D}{=}$ and we write $\Lambda \left( \Omega,\mathcal{F},P,(\mathcal{F}_t) \right)$ for the space of admissible integrands instead of $\Lambda$ if we want to stress the probability space on which the processes are defined.

\begin{Lemma}
\label{lem_unique_distribution_of_integral}
Let $\left(\Omega_i,\mathcal{F}_i,P_i,(\mathcal{F}_t^i)_{t\ge 0}\right)$ for $i=1,2$  be two filtered probability spaces,  $L_i$ two  cylindrical L\'evy processes with $L_i(t)\colon   U \to L^0\left(\Omega_i,\mathcal{F}_i,P_i\right)$ for $t\ge 0$, and  $\Psi_i \in \Lambda\left(\Omega_i,\mathcal{F}_i,P_i,(\mathcal{F}_t^i)\right)$. 
If the distributions of $(\Psi_1,L_1)$  and $(\Psi_2,L_2)$ satisfy
\begin{align}\label{eq.joint-distribution-equal}
&P\big(\Psi_1(s_1)\in B_1,\dots, \Psi_1(s_m)\in B_m,L_1(t_1)v_1\in C_1,\dots L_1(t_n)v_n\in C_n   \big) \notag \\
&\qquad = P\big(\Psi_2(s_1)\in B_1,\dots, \Psi_2(s_m)\in B_m,L_2(t_1)v_1\in C_1,\dots L_2(t_n)v_n\in C_n  \big)
\end{align}
for all $s_i,t_j\ge 0$, $B_i\in\Borel(L_{\rm HS}(U,H))$, $C_j\in\Borel(\R)$, $v_j\in U$ and $i=1,\dots, m$, $j=1,\dots , n$ then  
$$\int_0^T \Psi_1(s)\, \dL_1(s) \stackrel{\mathcal D}{=}  \int_0^T \Psi_2(s)\, \dL_2(s).$$
\end{Lemma}

\begin{proof}
We start by showing that for $0\leq t_1 < t_2 \leq T$  the random variables $\Psi_1(t_1) \big( L_1(t_2)-L_1(t_1) \big)$ and $\Psi_2(t_1) \big( L_2(t_2)-L_2(t_1) \big)$ are equal in distribution.
First, assume that the integrands are of the form
\begin{equation*}
\Psi_1(t_1) = \sum_{k=1}^{m_1} \1_{A_{1,k}}\psi_{1,k}, \qquad
\Psi_2(t_1) = \sum_{k=1}^{m_2} \1_{A_{2,k}}\psi_{2,k},
\end{equation*}
for $A_{i,k}\in {\mathcal{F}_{t_1}^i}$,  $\psi_{i,k}\in L_{\rm HS}(U,H)$ and $m_1,m_2\in\N$. We may assume that for each fixed $i\in\{1,2\}$ the operators $\psi_{1,1},\ldots,\psi_{1,m_1}$ are pairwise distinctive and similarly for $\psi_{2,1},\ldots,\psi_{2,m_2}$.  
Since the distributions of the integrals $\int \Psi_i\, \dL_i$ are invariant under changes on 
null sets, we can assume that $P(A_{i,k})>0$ for all $k=1,\dots, m_i$ and $i=1,2$. Since $\Psi_1(t_1)$ and $\Psi_2(t_1)$ have the same distribution it follows that 
$m_1=m_2$ and $\{\psi_{1,1},\dots, \psi_{1,m_1}\}= \{\psi_{2,1},\dots, \psi_{2,m_2}\}$. 
Without loss of generality, we may assume that $\psi_{1,1} = \psi_{2,1},\dots, \psi_{1,m_1} = \psi_{2,m_2}$.
The assumption on the joint distribution  guarantees
\begin{multline*}
P \Big( \1_{A_{1,k}} \in B_k, L_1(t_1)(\psi_{1,k}^\ast h)\in C_k, L_1(t_2)(\psi_{1,k}^\ast h) \in D_k \text{ for } k=1,\ldots,m_1 \Big) \\
=P\Big( \1_{A_{2,k}} \in B_k, L_2(t_1)(\psi_{2,k}^\ast h)\in C_k, L_2(t_2)(\psi_{2,k}^\ast h)\in D_k \text{ for } k=1,\ldots,m_2 \Big),
\end{multline*}
for all $B_k, C_k,D_k\in\Borel(\R)$ and $h\in H$, which implies
\begin{align*}
\langle \Psi_1(t_1)&\big(L_1(t_2)-L_1(t_1)\big), h\rangle =
\sum_{k=1}^{m_1} \1_{A_{1,k}}\big(L_1(t_2)-L_1(t_1)\big)(\psi_{1,k}^\ast h)\\
&\stackrel{\mathcal D}{=}
\sum_{k=1}^{m_2} \1_{A_{2,k}}\big(L_2(t_2)-L_2(t_1)\big)(\psi_{2,k}^\ast h)
=\langle \Psi_2(t_1)\big(L_2(t_2)-L_2(t_1)\big), h\rangle .
\end{align*}
Since the distribution of the scalar products $\langle \Psi_i(t_1)\big(L_i(t_2)-L_i(t_1)\big), h\rangle $ for all $h\in H$ uniquely determines the distribution of the $H$-valued random variables $\Psi_i(t_1)\big(L_i(t_2)-L_i(t_1)\big)$ we obtain equality in distribution of the latter. If $\Psi_i(t_1)$ is an arbitrary ${\mathcal F}_{t_1}^i$-measurable random variable then the $H$-valued random variable $\Psi_i(t_1)\big(L_i(t_2)-L_i(t_1)\big)$ is defined as the limit in probability of
$\Psi_i^n(t_1)\big(L_i(t_2)-L_i(t_1)\big)$ as $n\to\infty$, where $\Psi_i^n(t_1)$ are discretised versions of $\Psi_i(t_1)$ of the above form; see \cite[Th.\ 4.2]{Jakubowski_Riedle}.
Since the discretised versions also satisfy  \eqref{eq.joint-distribution-equal},  we have equality in distribution of the random variables $\Psi_i(t_1)\big(L_i(t_2)-L_i(t_1)\big)$  for $i=1,2$. 

Repeating the above arguments, but simultanously at times $0\le t_1<\dots <t_n\le T$, establishes 
equality in distribution of the $H^n$-valued random variables 
$$
\Big( \Psi_i(t_1)\big(L_i(t_2)-L_i(t_1)\big),\dots 
 ,\Psi_i(t_{n-1})\big(L_i(t_n)-L_i(n_{n-1})\big)
\Big)
\qquad\text{for }i=1,2.
$$
This finishes the proof in the case when the integrands $\Psi_i$ are simple, i.e.\ of the form
\begin{align*}
\Psi_i(t)= \Psi_{i,0} \1_{\{0\}}(t) + \sum_{k=1}^{n-1} \Psi_{i,k} \1_{(t_k,t_{k+1}]}(t)
\qquad\text{for }t\in [0,T],
\end{align*}
where $\Psi_{i,k}\colon\Omega_i\to  L_{\rm HS}(U,H)$ are ${\mathcal F}_{t_k}^i$-measurable random variables for $k=1,\dots, n-1$. 

For arbitrary, adapted, c\`agl\`ad processes $\Psi_i\in \Lambda\left(\Omega_i,\mathcal{F}_i,P_i,(\mathcal{F}_t^i)\right)$ for $i=1,2$, let $\Psi_i^n$ be discretised versions of $\Psi_i$ such that $(\Psi_i^n)_{n\in\N}$ converges to $\Psi_i$ in probability in the Skorokhod metric; see \cite[Prop.\ VI.6.37]{Jacod_Shiryaev}  or  \cite[Lem.\ 5.2]{Jakubowski_Riedle}. Since $\Psi_i^n$ are discretised versions of $\Psi_i$ it follows that these also satisfy \eqref{eq.joint-distribution-equal}. Consequently, the above argument shows that 
$$ \int_0^T \Psi_n^1(s) \dL_1(s) \stackrel{\mathcal D}{=}  \int_0^T \Psi_n^2(s) \dL_2(s).$$
The assertion follows by taking $n\to \infty$ according to \cite[Th.\ 5.1]{Jakubowski_Riedle}.
\end{proof}

\begin{Remark}\label{re.classical-cylindrical}
The cylindrical algebra ${\mathcal Z}(L_{\rm HS}(U,H)\times U)$ on the product space $L_{\rm HS}(U,H)\times U$ consists of sets of the form
\begin{align*}
\big\{ (\psi,u)\in L_{\rm HS}(U,H)\times U:\
\scapro{\psi}{\phi_1}+\scapro{u}{g_1}\in B_1,\dots, \scapro{\psi}{\phi_n}+\scapro{u}{g_n}\in B_n
  \big\}
\end{align*}
for some $\phi_i\in L_{\rm HS}(U,H)$,  
$g_i\in U$, $B_i\in\Borel(\R)$, $i=1,\dots, n$ and $n\in\N$. It is shown in \cite[p.\ 182]{Schwartz-Radon} that sets 
of the form 
\begin{align*}
\big\{ (\psi,u)\in L_{\rm HS}(U,H)\times U:\
\scapro{\psi}{\phi_1}\in B_1,\dots, \scapro{\psi}{\phi_m}\in B_m, 
\scapro{u}{g_1}\in C_1,\dots, \scapro{u}{g_n}\in C_n, 
  \big\}
\end{align*}
for some $\phi_i\in L_{\rm HS}(U,H)$,  $g_j\in U$, $B_i, C_j\in\Borel(\R)$ and $i=1,\dots, m$ and $j=1,\dots, n$ are co-final for the  cylindrical algebra ${\mathcal Z}(L_{\rm HS}(U,H)\times U)$. Thus, Condition \eqref{eq.joint-distribution-equal} implies that the cylindrical probability distributions of $\big(\big(\Psi_i(s_1), L_i(t_1)\big) ,\dots,\big(\Psi_i(s_m), L_i(t_m)\big)\big)$ for $i=1,2$ coincide on ${\mathcal Z}(L_{\rm HS}(U,H)\times U)\times \ldots \times {\mathcal Z}(L_{\rm HS}(U,H)\times U)$, i.e.\ the finite-dimensional distributions of the (cylindrical) processes $(\Psi_i,L_i)$ coincide on  ${\mathcal Z}(L_{\rm HS}(U,H)\times U)$. 

If $L_i$ is a genuine L\'evy process on $U$ for $i=1,2$ and the finite dimensional distributions of $(\Psi_i, L_i)$ for $i=1,2$ coincide on ${\mathcal Z}(L_{\rm HS}(U,H)\times U)$ it follows that these also coincide on the Borel $\sigma$-algebra $\Borel(L_{\rm HS}(U,H)\times U)$ since the latter is generated by the cylindrical algebra. Consequently, in this case Condition \eqref{eq.joint-distribution-equal} requires that the finite-dimensional distributions of the processes $(\Psi_i,L_i)$ coincide on $\Borel(L_{\rm HS}(U,H)\times U)$.
The latter condition is known to be equivalent to the requirement that the distributions of the stochastic processes $(\Psi_i,L_i)$ for $i=1,2$ coincide on the Borel $\sigma$-algebra of the product space of Skorokhod spaces with  c\`agl\`ad and  c\`adl\`ag functions, i.e.\ one would write $(\Psi_1,L_1)\stackrel{\mathcal D}{=}(\Psi_2,L_2)$. 
\end{Remark}

\subsection{Tightness criterion in the Skorokhod space}

We say that a sequence of $H$-valued c\`adl\`ag processes $(X_n)$ satisfies the Aldous Condition if for any $\epsilon, \eta >0$ there exists $\delta>0$ such that for all sequences of stopping times $(\tau_n)$ such that $\tau_n+\delta \leq T$ one has
$$\sup_{n\in \N} \sup_{0<\theta\leq \delta} P(\norm{X_n(\tau_n+\theta)-X_n(\tau_n)} \geq \eta)\leq \epsilon.$$
We formulate and prove another version of the well-known result, which says that if a sequence of c\`adl\`ag processes $(X_n)$ satisfies the Aldous Condition and for every fixed $t$ it attains values in compact sets with arbitrarily large probability then it is tight in the space $D([0,T];H)$. 
Detailed exposition of this method is presented in \cite[Sec.\ 3.7 and 3.8]{Ethier_Kurtz}. 
Criteria specifically useful in the context of SPDEs are given in \cite{Motyl_2013}, which we build upon here.

\begin{Theorem}
\label{th_my_compact_containment_and_ldous}
Let $(X_n)$ be a sequence of $H$-valued c\`adl\`ag adapted processes. Assume that there exists a compactly embedded  subspace $\Gamma$ of $H$ with norm $\norm{\cdot}_\Gamma$ such that
\begin{equation}
\label{version_of_compact_containment}
\forall \epsilon>0 \, \exists R>0 \, \forall t \in [0,T] \cap \Q, \, n \in \N : P\big( X_n(t) \in \Gamma \text{ and } \norm{X_n(t)}_\Gamma \leq R \big) \geq 1-\epsilon.
\end{equation}
If $(X_n)$ satisfies the Aldous Condition then $(X_n)$ is tight in $D([0,T];H)$.
\end{Theorem}

\begin{proof}
Fix $\epsilon>0$ and arrange $[0,T] \cap \Q$ in a sequence $(t_k)$. For every $k \in \N$ we choose $R_k>0$ such that for all $n \in \N$
$$P\left( X_n(t_k) \in \Gamma \text{ and } \norm{X_n(t_k)}_\Gamma \leq R_k \right) \geq 1 - \frac{\epsilon}{2^{k+1}}.$$
Let $B = \{ x \in D([0,T];H) : x(t_k) \in \Gamma \text{ and } \norm{x(t_k)}_\Gamma \leq R_k \text{ for all } k \in \N \}$.
Then
\begin{align}
\label{probability_of_B}
P( X_n \notin B) 
&= P \bigg( \bigcup_{k=1}^\infty \{ X_n(t_k) \in \Gamma \text{ and } \norm{X_n(t_k)}_\Gamma \leq R_k \}^c \bigg) \notag \\
&\leq \sum_{k=1}^\infty P(\{ X_n(t_k) \in \Gamma \text{ and } \norm{X_n(t_k)}_\Gamma \leq R_k \}^c) \notag \\
&\leq \sum_{k=1}^\infty \frac{\epsilon}{2^{k+1}} = \frac{\epsilon}{2}. 
\end{align}
By \cite[Lem.\ 7 and 8]{Motyl_2013} the Aldous Condition implies the existence of a measurable set $A \subset D([0,T];H)$ such that $P(X_n \in A) \geq 1-\tfrac{\epsilon}{2}$ and
\begin{equation}
\label{condition_ii_holds}
\lim_{\delta \to 0} \sup_{x \in A} \omega(x,\delta) = 0,
\end{equation}
where $\omega$ is the usual modulus of continuity in $D([0,T];H)$, see e.g.\ \cite[Ch.\ 3]{Billingsley}.
We show that the assumptions of Theorem 1 in \cite{Motyl_2013} are satisfied for the set $A\cap B$:
\begin{enumerate}
\item there exists a dense subset $\zeta\subset [0,T]$ such that for all $t \in \zeta$ the set $\{ x(t) : x \in A \cap B \}$ is relatively compact,
\item $\lim\limits_{\delta\to 0} \sup\limits_{x \in A\cap B} \omega(x, \delta)=0$.
\end{enumerate}
Note that Condition (i) holds with $\zeta=[0,T]\cap \Q$ because of the choice of $B$ and the fact that each closed ball $\{h\in \Gamma : \norm{h}_\Gamma \leq R_k\}$ is relatively compact in $H$ because of the assumed compact embedding of $\Gamma$. Condition (ii) is satisfied by \eqref{condition_ii_holds}. Thus, Theorem 1 in \cite{Motyl_2013}  implies that $A \cap B$ is relatively compact. Inequality \eqref{probability_of_B} implies 
$$P(X_n \in \overline{A \cap B}) \geq P(X_n \in A\cap B) \geq P(X_n \in A) - P(X_n \in B^c) \geq 1 - \frac{\epsilon}{2} - \frac{\epsilon}{2} = 1-\epsilon,$$ 
which proves tightness of $(X_n)$.
\end{proof}

\subsection{Continuity and measurability in the Skorokhod space}
\label{subsec_continuity_Skorokhod}

We prove a result concerning the composition of a strongly continuous semigroup and Hilbert-Schmidt operators. We consider continuity of the semigroup considered as a mapping on the Skorokhod space. 

\begin{Lemma}
\label{lem_uniform_continuity_compact_set}
Let $(S(t))_{t\ge 0}$ be a strongly continuous semigroup on $H$. 
If $K\subset L_{\rm HS}(U,H)$ is compact, then $\sup\limits_{\psi \in K} \HSnorm{(S(t)-\Id)\psi} \to 0$ as $t \to 0$.
\end{Lemma}

\begin{proof}
We first show that if $\psi \in L_{\rm HS}(U,H)$, then $\norm{(\Id-S(t))\psi}_{L_{\rm HS}(U,H)}\to 0$ as $t\to 0$.
Let $(e_n)$ be an orthonormal basis of $U$.
We have
$$\HSnorm{(\Id-S(t))\psi}^2 
= \sum_{n=1}^\infty \norm{(\Id-S(t))\psi e_n}^2.$$
By the strong continuity of the semigroup each term in the sum converges to $0$.
Let $M:= \sup_{t \in [0,T]} \norm{S(t)}$. We have the bound $\norm{(\Id-S(t))\psi e_n}^2 \leq (1+M)^2 \norm{\psi e_n}^2$, which is summable because $\psi$ is Hilbert-Schmidt.
The application of the Lebesgue dominated convergence theorem proves the claim.

Now let $K\subset L_{\rm HS}(U,H)$ be compact and fix $\epsilon>0$. In the proof we use a method similar to the proof of \cite[Th.\ 2.3.2]{Pazy}.
Take $\epsilon_1 = \frac{\epsilon}{2(1+M)}$ and choose a covering of $K$ consisting of the balls $B(\psi_i,\epsilon_1)$ for $i=1,\ldots ,N$ with centres $\psi_i$ and radius $\epsilon_1$ for some $\psi_i \in L_{\rm HS}(U,H)$.
There exists $\delta$ such that for $s\leq \delta$ and $i=1,\ldots,N$ we have $\HSnorm{(\Id-S(s))\psi_i} \leq \frac{\epsilon}{2}$. 
For any $s\leq \delta$ and $\psi \in K$ we find the closest center $\psi_i$ and estimate
\begin{align*}
\HSnorm{(\Id-S(s))\psi}
&\leq \HSnorm{(\Id-S(s))\psi_i} + \HSnorm{(\Id-S(s))(\psi-\psi_i)} \\
&\leq \frac{\epsilon}{2} + (1+M)\epsilon_1 = \epsilon. \qedhere
\end{align*}
\end{proof}

\begin{Lemma}
\label{lem_continuity_in_Skorokhod_space}
Let $(S(t))_{t\ge 0}$ be a strongly continuous semigroup on $H$ and $G\colon H \to L_{\rm HS}(U,H)$ a continuous map. 
If there exists a constant $c_L>0$ such that 
$$\HSnorm{S(t)(G(x)-G(y))} \leq c_L \norm{x-y}
\qquad\text{for  all } t \in [0,T], x,y \in H, $$
then the mapping
\begin{equation*}
\Theta \colon D([0,T];H) \to D([0,T];L_{\rm HS}(U,H)), \qquad \Theta(x)(s):= S(T-s)G(x(s))
\end{equation*}
is continuous.
\end{Lemma}

\begin{proof}
We use the notation $\norm{x}_\infty = \sup_{s \in [0,T]} \norm{x(s)}$ when $x \in D([0,T];H)$.
Recall that the Skorokhod topology on $D([0,T];H)$ is induced by the metric
$$d(x,y) = \inf_{j \in \Pi} \left( \norm{\iota-j}_\infty \vee \norm{x-y\circ j}_\infty \right),$$
where $\Pi$ is the set of all increasing bijections of $[0,T]$ and $\iota\colon [0,T] \to [0,T]$ is the identity mapping, see \cite[p.\ 124]{Billingsley}.
By the elementary inequality
$a \vee b \leq a+b \leq 2(a\vee b)$ for all $ a,b \geq 0,$
we see that the metric $d$ is equivalent to the following one:
$$d^+(x,y) := \inf_{j \in \Pi} \left( \norm{\iota-j}_\infty + \norm{x-y\circ j}_\infty  \right), \qquad x,y \in D([0,T];H).$$
Fix $x_n,x \in D([0,T];H)$ such that $(x_n)$ converge to $x$ and let $\epsilon>0$.
By \cite[Prob.\ 1, p.\ 146]{Dieudonne}, the image of $[0,T]$ by $x$ is relatively compact in $H$.
By the continuity of $G$, the set $K := \{G(x(s)) : s \in [0,T]\}$ is relatively compact in $L_{\rm HS}(U,H)$. 
Let $M:= \sup_{t \in [0,T]} \norm{S(t)}$.
By Lemma \ref{lem_uniform_continuity_compact_set} we obtain that for some $\delta>0$
\begin{equation}
\label{sup_sup}
\sup_{s\leq \delta} \sup_{\psi \in K} \HSnorm{(\Id-S(s))\psi} 
\leq \frac{\epsilon}{2M}.
\end{equation}
Without loss of generality we assume that $\delta \leq \frac{\epsilon}{c_L+1}$.
There exists $n_0$ such that for all $n\geq n_0$ we have $d^+(x_n,x) \leq \tfrac{\delta}{2}$.
By definition of the metric, for each $n\geq n_0$ there exists $j_n \in \Pi$ such that
$$\norm{\iota-j_n}_\infty  + \norm{x_n-x\circ j_n}_\infty 
\leq d^+(x_n,x) + \frac{\delta}{2}.$$
Thus for each $n\geq n_0$ there exists $j_n \in \Pi$ such that
\begin{equation}
\label{less_than_delta}
\norm{\iota-j_n} + \norm{x_n-x\circ j_n}_\infty \leq \delta.
\end{equation}
By the semigroup property and the assumed continuity we obtain 
\begin{align}
\label{Skorokhod_distance_estimate}
&d^+(\Theta(x_n),\Theta(x)) \notag \\
&\leq \norm{\iota-j_n}_\infty + \sup_{s \in [0,T]} \norm{S(T-s) G(x_n(s)) - S(T-j_n(s))G(x(j_n(s)))}  \notag \\
&\leq \norm{\iota-j_n}_\infty + \sup_{s \in [0,T]} \HSnorm{S(T-s)\left(G(x_n(s))-G(x(j_n(s)))\right)} \notag\\
&\quad + \sup_{s \in [0,T]} \HSnorm{\left(S(T-s)-S(T-j_n(s))\right) G(x(j_n(s)))}  \\
&\leq \norm{\iota-j_n}_\infty + c_L \sup_{s \in [0,T]} \norm{x_n(s)-x(j_n(s))}  \notag\\
&\quad \! + \!\! \sup_{s\in [0,T]} \HSnorm{S(T-(j_n(s)\vee s)) \big(S(j_n(s)\vee s - s) \! - \! S(j_n(s) \vee s -j_n(s))\big)G(x(j_n(s))}. \notag
\end{align}
By \eqref{less_than_delta} we have $\abs{s-(j_n(s)\vee s)}\leq \delta$ for all $n\geq n_0$ and $s\in [0,T]$. 
Note that
$$S((j_n(s)\vee s)-s)-S((j_n(s)\vee s)-j_n(s))=
\begin{cases}
\Id - S(s-j_n(s)), & \text{if } j_n(s) \leq s, \\
S(j_n(s)-s)-\Id, & \text{if } j_n(s)>s.
\end{cases}$$
Consequently inequality \eqref{sup_sup} implies that 
\begin{align*}
&\sup_{s\in [0,T]} \HSnorm{S(T-(j_n(s) \vee s))\big(S((j_n(s)\vee s)-s)-S((j_n(s)\vee s)-j_n(s))\big) G(x(j_n(s))}  \\
&\leq M \sup_{s\in [0,T]} \HSnorm{\big(S((j_n(s)\vee s)-s)-S(j(_n(s)\vee s)-j_n(s))\big) G(x(j_n(s))}  \\
&\leq M \sup_{s \leq \delta} \sup_{\psi \in K} \HSnorm{(\Id-S(s))\psi} \\
&\leq \frac{\epsilon}{2}.
\end{align*}
Applying this estimate together with \eqref{less_than_delta} in  \eqref{Skorokhod_distance_estimate} 
establishes for each  $n \geq n_0$ that
\begin{equation*}
d^+(\Theta(x_n),\Theta(x)) 
\leq \delta + c_L \delta+ \frac{\epsilon}{2}
= (c_L+1)\delta + \frac{\epsilon}{2}
\leq \frac{\epsilon}{2} + \frac{\epsilon}{2}
=\epsilon. \qedhere
\end{equation*}
\end{proof}

We finish this subsection by establishing that every cylindrical L\'evy process can be viewed as a random variable taking values in a Polish space.
First, we recall some facts about measurability of evaluations with respect to the Skorokhod topology, see \cite[Lem.\ VI.3.12, Lem.\ VI.3.14]{Jacod_Shiryaev}.
For a c\`adl\`ag process $X$ let $\mathcal{J}(X) := \{ t \geq 0 : P(\Delta X(t) \neq 0)>0 \}$. The set $\mathcal{J}(X)$ is at most countable.
Recall that for each $t \in [0,T]$ the function $\pi_t \colon  D([0,T];H) \to H$ defined by $\pi_t(x)=x(t)$ is measurable, see \cite[Th.\ 12.5(ii)]{Billingsley}.
Thus if $X$ and $Y$ are two $H$-valued processes with the same distribution in $D([0,T];H)$, then $X(t)$ and $Y(t)$ have the same distribution for any $t \in [0,T]$.


%
%

Let $(u_n)$ be a countable dense sequence in $U$. 
Let $\R^\infty = \prod\limits_{n=1}^\infty \R$.
With each cylindrical L\'evy process $L=(L(t) : t \geq 0)$ we associate a mapping
\begin{equation}
\label{definition_A_L}
A_L\colon \Omega \to D([0,T];\R^\infty), \qquad A_L(\omega)(t) := \left( (L(t)u_n)(\omega) \right)_{n \in \N}.
\end{equation}
Vice versa, the mapping $A_L$ determines the cylindrical L\'evy process:
\begin{Lemma}
\label{lem_metric_space}
The space $D([0,T];\R^\infty)$ is a separable and complete metric space. The function $A_L$ is measurable.
Moreover, $A_L$ uniquely determines the process $L$.
\end{Lemma}

\begin{proof}
For the first statement see \cite[Cor.\ 2.3.16, Th.\ 4.2.2, Th.\ 4.3.12]{Engelking} and \cite[Lem.\ 1.6(iii)]{Jakubowski_1986}.
We prove that $A_L$ is measurable. 
For every $t\in [0,T]$ the mapping $\omega \mapsto \left(L(t)u_n\right)_{n \in \N}$ is measurable. 
By \cite[Cor.\ 2.4]{Jakubowski_1986} this implies measurability of $A_L$.

Note that $A_L$ uniquely determines $L$ as for $u \in U \setminus \{u_n: n \in \N \}$ we find a subsequence $(u_{n_k})_{k\in\N}$ converging to $u$ and then by the continuity of $L(t)$ we have that $L(t)u =  \lim\limits_{k\to \infty} A_L(t)(u_{n_k})$ , where the limit is taken in probability. 
\end{proof}

\section{Stochastic evolution equation}
\label{sec_SPDE}

From now on we assume $\alpha \in (1,2)$. We consider the problem of existence and uniqueness of solutions for the evolution equation
\begin{equation}
\label{SPDE_stable}
\dX(t) = \left( AX(t) + F(X(t))\right) \dt + G(X(t-)) \dL(t)
\end{equation}
with an $\mathcal{F}_0$-measurable initial condition $X(0)=X_0$. Here, $A$ is a generator of a strongly continuous semigroup $S$ on the Hilbert space $H$, and $F\colon H \to H$ and $G\colon H \to L_{\rm HS}(U,H)$ are some mappings. The driving noise $L$ is assumed to be a canonical $\alpha$-stable cylindrical L\'evy process on the Hilbert space $U$ for $\alpha \in (1,2)$.

\begin{Definition}
A filtered probability space $(\Omega,\mathcal{F},P,(\mathcal{F}_t))$, a cylindrical L\'evy process $L$ and an $H$-valued, c\`adl\`ag process $X$ is a \emph{weak mild solution} to \eqref{SPDE_stable} if
\begin{equation}
\label{SPDE_mild}
X(t) = S(t)X_0 + \int_0^t S(t-s)F(X(s)) \ds + \int_0^t S(t-s) G(X(s-)) \dL(s)
\end{equation}
holds $P$-a.s.\ for all $t \in [0,T]$. 
If such a c\`adl\`ag process $X$ exists for any given probability basis and any given $\alpha$-stable cylindrical L\'evy process $L$, we say that there is a \emph{strong mild solution}.
We say that \emph{pathwise uniqueness} holds if for any two c\`adl\`ag processes $X_1$ and $X_2$ defined on the same probability space and satisfying \eqref{SPDE_mild} we have
$$P(X_1(t) = X_2(t) \text{ for all } t \in [0,T]) = 1.$$
\end{Definition}

We list all the assumptions which will guarantee the existence and uniqueness of solutions:
\begin{enumerate}[label=(A\arabic*)]
\item \label{it_contractions} The semigroup $S$ satisfies the following additional conditions:
\begin{enumerate}
\item[(i)] $S$ is a semigroup of contractions, i.e.\ $\norm{S(t)} \leq 1$ for all $t\geq 0$,
\item[(ii)] $0$ is in the resolvent set $\rho(A)$,
\item[(iii)] $\{ \lambda \in \C : 0<\omega<\abs{\arg \lambda} \leq \pi\} \subset \rho(-A)$ for some $\omega<\frac{\pi}{2}$.
\item[(iv)] The embedding $D(A) \subset H$ is compact.
\end{enumerate}
\item \label{it_boundedness_fractional} There is $\delta>0$ and $M_0\geq 0$ such that for all $t \in (0,T]$ and $x \in H$:
\begin{align*}
\norm{S(t)F(x)}_{D((-A)^\delta)} &\leq M_0, \\
\norm{S(t)G(x)}_{L_{\rm HS}(U,D((-A)^\delta))} &\leq M_0.
\end{align*}
\item \label{it_Lipschitz} There exist constants $C_F$ and $C_G$ such that for all $x,y \in H$ and $t\in [0,T]$:
\begin{align*}
\norm{S(t)(F(x)-F(y))} &\leq C_F \norm{x-y},\\
\HSnorm{S(t)(G(x)-G(y))} &\leq C_G \norm{x-y}.
\end{align*}
\end{enumerate}

\begin{Theorem}
\label{th_existence_weak_sol_main_result}
Under assumptions {\upshape{\ref{it_contractions}}--\upshape{\ref{it_Lipschitz}}} equation \eqref{SPDE_mild} has a unique mild solution.
\end{Theorem}

The reminder of the paper is devoted to the proof of this result. After numerous preparatory lemmas in Subsections \ref{subsec_pathwise_uniqueness}-\ref{subsec_estimating_the_norm_of_difference} we complete the proof of existence and uniqueness in Subsection \ref{subsec_proof_main_result}. The outline is as follows.
We first prove pathwise uniqueness and existence of a weak mild solution on the interval $[0,T]$ for $T$ sufficiently small. A generalisation of the Yamada--Watanabe theorem implies the existence of a solution in the strong sense. Finally, we establish existence of the strong solution on any interval of arbitrary length $T$.


\begin{Remark}
The requirements imposed on a semigroup in (A1) are standard and are satisfied by most analytic semigroups, e.g.\ the heat semigroup on a bounded interval. 
The Lipschitz Condition (A3) is also standard, see e.g.\ Peszat and Zabczyk \cite[p.\ 142]{Peszat_Zabczyk}). Condition (A2) is more restrictive than the usual linear growth condition, see  
the following Remark \ref{rem_boundedness_fractional_is_stronger}. One could certainly weaken this condition by using stopping time arguments, but for the sake of a clearer presentation without distracting technical arguments we will not apply this here. For example, condition similar to (A2) is assumed in the work 
 \cite{Jentzen_Rockner} by Jentzen and R{\"o}ckner; as they consider Gaussian perturbations, their condition can be formulated in terms of the reproducing kernel Hilbert space, allowing for a slightly 
weaker version. 

A simple example for coefficients $F$ and $G$ satisfying Condition (A2) are diagonal operators along 
the eigenbasis of the generator $A$, a situation extensively considered in Lototsky and Rozovsky \cite{Lototsky_Rozovsky}. 
For this purpose, let $(h_n)$  be an orthonormal basis of $H$ consisting of eigenvectors of $A$, that is $A h_n = -\lambda_n h_n$ for all $n\in\N$ and some  $\lambda_n>0$. Then $S(t)h_n = e^{-\lambda_n t} h_n$. 
Let $g_n\colon H\to U$ be some functions such that $(g_n(x))_{n\in\N}$ is weakly square-summable for all $x\in H$, in which case we can define 
$$G(x)u = \sum_{n=1}^\infty \langle g_n(x),u \rangle h_n\qquad \text{for all }x\in H, \, u\in U.$$
Then $G(x)^*h_n = g_n(x)$ for $n\in \N$.
Since
$$\norm{S(t)G(x)}_{L_{\rm HS}(U,D((-A)^\delta))}^2 
= \sum_{n=1}^\infty \lambda_k^{2\delta} e^{-2\lambda_k t} \norm{g_n(x)}_U^2,$$
condition (A2) holds provided the speed of convergence of $(g_n(x))_{n\in\N}$ to $0$ is large enough.
\end{Remark}


\begin{Remark}
\label{rem_boundedness_fractional_is_stronger}
Each operator $\psi \in L_{\rm HS}(U,D((-A)^\delta))$ satisfies
$$\HSnorm{\psi}^2
\leq \bignorm{(-A)^{-\delta}}^2 \sum_{n=1}^\infty \norm{\psi e_n}_{D((-A)^\delta)}^2
= \bignorm{(-A)^{-\delta}}^2 \norm{\psi}_{L_{\rm HS}(U,D((-A)^\delta))}^2.$$
Thus, \ref{it_boundedness_fractional} implies the following: 
\begin{enumerate}[label=(A\arabic*'),start=2]
\item \label{it_boundedness} There exists a constant $M_0>0$ such that for all $t \in (0,T]$ and $x \in H$:
$$\norm{S(t)F(x)} \leq M_0, 
\qquad \text{and} \qquad 
\HSnorm{S(t)G(x)} \leq M_0.$$
\end{enumerate}
\end{Remark}

\begin{Remark}
\label{rem_Holder_VS_Lipschitz}
The boundedness Condition \ref{it_boundedness} implies that the Lipschitz Condition \ref{it_Lipschitz} is equivalent to the H\"older condition:
\begin{enumerate}[label=(A\arabic*'),start=3]
\item \label{it_Holder} There exist and constants $c_F,c_G \geq 0$ such that for any $q \in (0,1)$, $t \in [0,T]$ and for all $x,y \in H$:
\begin{align*}
\norm{S(t)(F(x)-F(y))} &\leq c_F \norm{x-y}^q,\\
\HSnorm{S(t)(G(x)-G(y))} &\leq c_G \norm{x-y}^q.
\end{align*}
\end{enumerate}
The implication \ref{it_Holder}$\implies$\ref{it_Lipschitz} can be seen by taking $q\nearrow 1$.
The converse implication \ref{it_Lipschitz}$\implies$\ref{it_Holder} follows by standard arguments. From the fact that, if  $\norm{x-y}\leq 1$, then $\norm{x-y}\leq \norm{x-y}^q$, and,  if $\norm{x-y}>1$, then  
$\norm{S(t)(F(x)-F(y))} \leq M_0 \leq M_0 \norm{x-y}$
due to \ref{it_boundedness}. Thus we may take $c_F = C_F \vee M_0$. The claim for $G$ follows analogously. 
\end{Remark}


\subsection{Pathwise uniqueness}
\label{subsec_pathwise_uniqueness}

We prove pathwise uniqueness using a special version of Gronwall's inequality from Theorem 2 in \cite{Willett_Wong}:

\begin{Theorem}[Willet, Wong, 1964]
\label{th_Willet_Wong}
Suppose that the functions $u, v, w\colon [0,T]\to \Rp$ are such that 
$v$, $w$ $vu$ and $wu^p$ are integrable for some $p\in [0,1)\cup(1,\infty)$. If
$$u(t) \leq \int_{0}^t v(s)u(s) \ds + \int_0^t w(s)u^p(s) \ds \qquad\text{for all } t \in [0,T],$$
then it follows  for $q:=1-p$ and all $t\in [0,T]$ that
$$u(t) \exp\left(-\int_0^t v(s) \ds \right) \leq \left( q\int_0^t w(s) \exp \left( -q\int_0^s v(r) \dr \right) \ud s \right)^{1/q}.$$
\end{Theorem}

\begin{Proposition}
\label{pro_pathwise_uniqueness}
With the constant $c_{2,\alpha}$ defined in Theorem \ref{th_tail_estimate} let 
\begin{equation*}
c_3:= \sup_{p\in (1,\alpha)} 2^{p-1} \left( \tfrac{\alpha-1}{\left( c_{2,\alpha}^{1/\alpha} \wedge c_{2,\alpha}\right) \alpha} c_F^p + c_G^p \right). 
\end{equation*}
If $T<\min\{1,\tfrac{1}{\alpha c_{2,\alpha} (c_3^\alpha \vee c_3)}\}$, then 
Condition \ref{it_Holder} implies that the pathwise uniqueness holds for equation \eqref{SPDE_stable}. More generally, if $X$ and $Y$ are two solutions with initial conditions $X_0$ and $Y_0$ respectively, then $X$ and $Y$ are indistinguishable processes on $\{ X_0=Y_0 \}$.
\end{Proposition}

\begin{proof}
Suppose that $X$ and $Y$ are two solutions of \eqref{SPDE_stable} with initial conditions $X_0$ and $Y_0$, respectively.
Let $\tau_n = \inf \left\{ t \geq 0 : \norm{X(t)} \wedge \norm{Y(t)} \geq n \right\} \wedge T$.
It follows by standard arguments as in \cite[Lem.\ 2.3.9]{Prevot_Rockner} and using the uniform convergence in probability as in the proof of Theorem \ref{th_tail_estimate} above that for any stopping time $\tau$ and for any $\Psi \in \Lambda$ we have
$$\int_0^{t\wedge \tau} \Psi(s) \dL(s) = \int_0^t \Psi(s) \1_{\{s \leq \tau\}} \dL(s).$$
It follows that 
\begin{align*}
(X(t \! \wedge \! \tau_n) \! - \! Y(t \! \wedge \! \tau_n))\1_{\{X_0=Y_0\}} 
&\!= \int_0^t \!\! S(t-s)\left( F(X(s)))\! - \! F(Y(s)) \right) \1_{\{s \leq \tau_n\} \cap \{X_0=Y_0\}} \ds \\
& \,\, + \!\! \int_0^t \!\! S(t-s)(G(X(s-)) \! - \!G(Y(s-))) \1_{\{s \leq \tau_n\} \cap \{X_0=Y_0\}} \ud L(s).
\end{align*}
H\"older's inequality implies for any measurable function $f\colon \Omega \times [0,t] \to H$ and $\alpha>p\geq 1$ that 
\begin{equation*}
\E \left[ \norm{\int_0^t f(s) \ds}^p \right]
\leq t^{p-1} \E \left[ \int_0^t \norm{f(s)}^p \ds \right]
\le t^{p-\frac{p}{\alpha}} \left(\E \left[ \int_0^t \norm{f(s)}^\alpha \ds \right] \right)^{p/\alpha}.
\end{equation*}
By applying Corollary \ref{pro_moment_inequality_2} and Condition \ref{it_Holder} for $q=\tfrac{p}{\alpha}$ we obtain that
\begin{align}
\label{moment_of_difference_in_uniqueness}
&\E \left[ \norm{X(t)-Y(t)}^p \1_{\{t < \tau_n\} \cap \{X_0=Y_0\}} \right] \notag \\
&\leq \E \left[ \norm{X(t\wedge \tau_n)-Y(t\wedge \tau_n)}^p \1_{\{X_0=Y_0\}} \right] \notag \\
&\leq 2^{p-1} \E \left[ \norm{\int_0^t S(t-s) \left( F(X(s))-F(Y(s)) \right) \1_{\{s \leq \tau_n\} \cap \{X_0=Y_0\}} \ds}^p \right] \notag \\
&\quad \! + 2^{p-1} \E \left[ \norm{\int_0^t S(t-s) \left( G(X(s-))-G(Y(s-)) \right) \1_{\{s \leq \tau_n\} \cap \{X_0=Y_0\}} \dL(s)}^p \right] \notag \\
&\le 2^{p-1} T^{p-\frac{p}{\alpha}} \left( \E \left[ \int_0^t \norm{S(t-s) \left( F(X(s))-F(Y(s)) \right)}^\alpha \1_{\{s \leq \tau_n\} \cap \{X_0=Y_0\}} \ds \right] \right)^{p/\alpha} \notag \\
&\quad \! + 2^{p-1} C_{\alpha,p} \left( \E \left[ \int_0^t \HSnorm{S(t-s)\left( G(X(s-))-G(Y(s-)) \right)}^\alpha \1_{\{s \leq \tau_n\} \cap \{X_0=Y_0\}} \ds \right] \right)^{p/\alpha}  \notag \\
&= 2^{p-1} \left( T^{p-\frac{p}{\alpha}} c_F^p + C_{\alpha,p} c_G^p\right) \left( \E \left[ \int_0^t \norm{X(s)-Y(s)}^p \1_{\{s < \tau_n\} \cap \{X_0=Y_0\}} \ds \right] \right)^{p/\alpha} \notag\\
&\leq  C_{\alpha,p} c_3 \left( \E \left[ \int_0^t \norm{X(s)-Y(s)}^p \1_{\{s < \tau_n\} \cap \{X_0=Y_0\}} \ds \right] \right)^{p/\alpha}, 
\end{align}
where $c_3:= 2^{p-1} \left( \tfrac{\alpha-1}{(c_{2,\alpha}^{1/\alpha} \wedge c_{2,\alpha}) \alpha} c_F^p + c_G^p\right)$. Note, the last inequality follows from the fact that $T<1$ and 
the very definition of $C_{\alpha,p}$ in Corollary \ref{pro_moment_inequality_2}  imply 
\begin{align*}
 T^{p-\frac{p}{\alpha}} c_F^p + C_{\alpha,p} c_G^p
\le \tfrac{1}{C_{\alpha,p}}(c_F^p + c_G^p)
\le  \tfrac{\alpha-1}{(c_{2,\alpha}^{1/\alpha} \wedge c_{2,\alpha}) \alpha}(c_F^p + c_G^p). 
\end{align*}
Raising inequality \eqref{moment_of_difference_in_uniqueness} to the power $\frac{\alpha}{p}$
and applying Theorem \ref{th_Willet_Wong} (with $v = 0$, $w = C_{\alpha,p}^{\alpha/p} c_3^{\alpha/p}$,  $\frac{p}{\alpha}$ instead of $p$ and $q= \frac{\alpha-p}{\alpha}$) show that 
\begin{align*}
 \left(\E \left[ \norm{X(t)-Y(t)}^p \1_{\{t < \tau_n\} \cap \{X_0=Y_0\}} \right] \right)^{\alpha/p} 
&\leq \left( \tfrac{\alpha-p}{\alpha} \int_0^t  C_{\alpha,p}^{\alpha/p} c_3^{\alpha/p} \ds \right)^\frac{\alpha}{\alpha-p}\\
&\leq (\alpha-p)^{-2} \left( \alpha c_{2,\alpha} (c_3^\alpha \vee c_3) t \right)^\frac{\alpha}{\alpha-p}, 
\end{align*}
where we used in the last inequality the inequalities $(\alpha-p)^{-\frac{\alpha}{p}} \leq (\alpha-p)^{-2}$ and $\alpha^{-1+\frac{\alpha}{p}} \leq \alpha$. 
Taking $p\nearrow \alpha$, using the assumed relation $ \alpha c_{2,\alpha} (c_3^\alpha \vee c_3) t < 1$  and applying the L'H\^{o}spital rule twice we derive
\begin{equation*}
\lim_{p \nearrow \alpha} \E \left[ \norm{X(t)-Y(t)}^p \1_{\{t < \tau_n\} \cap \{X_0=Y_0\}} \right] 
= 0.
\end{equation*}
Consequently, for each $\epsilon \in (0,1)$ there exists $1<p<\alpha$ such that
$$\E \left[ \norm{X(t) - Y(t)}^p \1_{\{t < \tau_n\} \cap \{X_0=Y_0\}} \right] 
\leq \epsilon^{2p}.$$ 
 Markov's inequality implies that 
$$P(\norm{X(t)-Y(t)} \geq \epsilon, t< \tau_n, X_0=Y_0) 
\leq \tfrac{1}{\epsilon^p} \E \left[ \norm{X(t)-Y(t)}^p \1_{\{t < \tau_n\} \cap \{X_0=Y_0\}} \right]
\leq \epsilon^p.$$
Since $\epsilon$ was arbitrary it follows that $X(t)$ and $Y(t)$ coincide on $\{t < \tau_n\} \cap \{ X_0=Y_0 \}$. 
This proves that $X$ and $Y$ are modifications of each other on $\{X_0=Y_0\}$. Since they are both c\`adl\`ag, it follows that $X$ and $Y$ are indistinguishable on $\{X_0=Y_0\}$.
\end{proof}

\subsection{Proof of tightness}
\label{subsec_proof_of_tightness}

\begin{Lemma}
\label{lem_tightness_of_Picard}
Consider the Picard approximating sequence defined by $X_0(t) = S(t)X_0$ and for $n \geq 1$ by
\begin{equation}
\label{Picard_teration}
X_n(t) = S(t)X_0 + \int_0^t S(t-s)F(X_{n-1}(s))\ds + \int_0^t S(t-s)G(X_{n-1}(s-)) \dL(s).
\end{equation}
If the assumptions \ref{it_contractions}, \ref{it_boundedness_fractional}, \ref{it_Lipschitz} hold, then the sequence $(X_n)$ is tight in $D([0,T];H)$.
\end{Lemma}

\begin{proof}
We prove the claim in 3 steps. Firstly, we show that $X_n$ is well-defined and c\`adl\`ag, secondly we verify the Aldous Condition and finally we verify Condition \eqref{version_of_compact_containment}. Then Theorem \ref{th_my_compact_containment_and_ldous} implies tightness of $(X_n)$. 

Step 1.
Condition \ref{it_contractions} implies by the dilation theorem \cite[Th.\ I.8.1]{Foias_Nagy}, that there exist a Hilbert space $H\subset \hat{H}$ with continuous embedding $i\colon H\to \hat{H}$ 
 and a group of unitary operators $(\hat S(t))_{t\in \R}$ such that $S(t) = P\hat{S}(t)$ for $t\geq 0$, where $P\colon  \hat{H} \to H$ denotes the canonical 
projection. 
We now prove by induction that $X_n$ is well defined and has a c\`adl\`ag modification. Suppose that $X_{n-1}$ is c\`adl\`ag and write
$$X_n(t) = S(t)X_0 + P \hat{S}(t) \int_0^t \hat{S}(-s)iF(X_{n-1}(s))\ds +  P \hat{S}(t) \int_0^t \hat{S}(-s)iG(X_{n-1}(s-))\dL(s).$$
The Lebesgue integral is clearly well defined by the continuity of $F$, which follows from \ref{it_Lipschitz} with $t=0$. Thus, the second term on the right-hand side is a continuous process.
We prove that that the integrand 
\begin{equation}
\label{integrand}
s \mapsto \hat{S}(-s)iG(X_{n-1}(s-))
\end{equation}
is c\`agl\`ad. 
Fix $s_0 < s \leq t$. We show that the right limit exists. We have
\begin{align}
\label{right_limit_estimate}
&\norm{\hat{S}(-s)iG(X_{n-1}(s-))-\hat{S}(-s_0)iG(X_{n-1}(s_0))}_{L_{\rm HS}(U,\hat{H})} \notag \\
&\qquad \qquad \leq \norm{\left( \hat{S}(-s)-\hat{S}(-s_0) \right) i G(X_{n-1}(s-))}_{L_{\rm HS}(U,\hat{H})} \notag \\
&\qquad \qquad \quad + \norm{\hat{S}(-s_0) i \left( G(X_{n-1}(s-))-G(X_{n-1}(s_0)) \right)}_{L_{\rm HS}(U,\hat{H})}.
\end{align}
The second term on the right-hand side tends to $0$ as $s\searrow s_0$ by \ref{it_Lipschitz} applied with $t=0$ and the right-continuity of $X_{n-1}$.
For fixed $\omega \in \Omega$ it follows from the continuity of $G$ guaranteed by \ref{it_Lipschitz} with $t=0$ and from \cite[Prob.\ 1, p.\ 146]{Dieudonne} that the set $\{ iG(X_{n-1}(s-)(\omega)) : s \in [0,t] \}$ is relatively compact in $L_{\rm HS}(U,\hat{H})$. Therefore, by Lemma \ref{lem_uniform_continuity_compact_set} applied for $\hat{H}$ instead of $H$ the first term on the right-hand side of \eqref{right_limit_estimate} tends to $0$ as well.
Similarly, one shows that the limit of $\hat{S}(-s)G(X_{n-1}(s-))$ from the left equals $\hat{S}(-s_0)G(X_{n-1}(s_0-))$.
This proves that the mapping \eqref{integrand} is c\`agl\`ad. Thus, the integral in \eqref{Picard_teration} is well defined according to \cite{Jakubowski_Riedle} and  $X_n$ has a c\`adl\`ag modification. 

Step 2.
Fix $\epsilon, \eta >0$. For any $\theta$ and a sequence of stopping $(\tau_n)$ times such that $\tau_n + \theta \leq T$ we have
\begin{equation*}
\label{long_tranformation}
\begin{aligned}
&X_n(\tau_n+\theta)-X_n(\tau_n) \\
&= S(\tau_n) \left( S(\theta)- \Id \right)X_0 + \int_0^{\tau_n} S(\tau_n-s) \left( S(\theta) - \Id\right) F(X_{n-1}(s)) \ds \\
&\quad + \int_0^{\tau_n} S(\tau_n-s)\left( S(\theta) - \Id \right) G(X_{n-1}(s-))\dL(s) + \int_{\tau_n}^{\tau_n+\theta} S(\tau_n+\theta-s) F(X_{n-1}(s)) \ds \\
&\quad + \int_{\tau_n}^{\tau_n+\theta} S(\tau_n+\theta-s)G(X_{n-1}(s-))\dL(s).
\end{aligned}
\end{equation*}
Let $M:=\sup\limits_{t \in [0,T]} \norm{S(t)}$.
By using Markov's inequality and Theorem \ref{th_tail_estimate} we obtain
\begin{align*}
P(\Vert X_n(\tau_n&+\theta)-X_n(\tau_n) \Vert \geq \eta) \\
&\leq P \left( \norm{S(\tau_n)\left(S(\theta) - \Id\right) X_0} \geq \frac{\eta}{5} \right) \\
&\qquad + P\left( \norm{\int_0^{\tau_n} \left( S(\theta) - \Id \right) S(\tau_n-s) F(X_{n-1}(s)) \ds}\geq \frac{\eta}{5} \right)\\
&\qquad + P\left( \norm{\int_0^{\tau_n} (S(\theta) - \Id) S(\tau_n-s)G(X_{n-1}(s-)) \dL(s)} \geq \frac{\eta}{5} \right) \\
&\qquad + P\left( \norm{\int_{\tau_n}^{\tau_n+\theta} S(\tau_n+\theta-s) F(X_{n-1}(s)) \ds}\geq \frac{\eta}{5} \right) \\
&\qquad + P\left( \norm{\int_{\tau_n}^{\tau_n+\theta} S(\tau_n+\theta-s)G(X_{n-1}(s-)) \dL(s)} \geq \frac{\eta}{5} \right) \\
&\leq P \left( \norm{(S(\theta) - \Id)X_0} \geq \frac{\eta}{5M} \right) \\
&\qquad + \frac{5}{\eta} \E \left[ \norm{\int_0^T \1_{[0,\tau_n]}(s) \left(S(\theta) - \Id\right) S(\tau_n-s) F(X_{n-1}(s)) \ds} \right]  \\
&\qquad + \frac{5^\alpha c_{2,\alpha}}{\eta^\alpha} \E \left[ \int_0^T \1_{[0,\tau_n]}(s) \HSnorm{(S(\theta) - \Id)S(\tau_n-s)G(X_{n-1}(s-))}^\alpha \ds \right] \\
&\qquad + \frac{5}{\eta} \E \left[ \norm{\int_0^T \1_{[\tau_n,\tau_n+\theta]}(s) S(\tau_n+\theta-s) F(X_{n-1}(s)) \ds} \right] \\
&\qquad + \frac{5^\alpha c_{2,\alpha}}{\eta^\alpha} \E \left[ \int_{\tau_n}^{\tau_n+\theta} \HSnorm{S(\tau_n+\theta-s)G(X_{n-1}(s-))}^\alpha \ds \right] \\
&=: e_1 + e_2 + e_3 + e_4 + e_5.
\end{align*}
Strong continuity of $(S(t))$ shows that the term $e_1$ tends to $0$ as $\theta$ tends to $0$.
For $e_2$ we use the estimate 
\begin{align*}
&\norm{(S(\theta) - \Id)S(\tau_n-s)F(X_{n-1}(s))} \\
&\qquad\qquad \leq \norm{S(\theta)-\Id}_{L(D((-A)^\delta),H)} \norm{S(\tau_n-s)F(X_{n-1}(s))}_{D((-A)^\delta)}.
\end{align*}
From \eqref{norm_continuity_of_semigroup} and \ref{it_boundedness_fractional} we conclude $e_2 \leq C\theta^\delta M_0$.
By similar arguments as for $e_2$ we derive  
$e_3 \leq \tfrac{5^\alpha c_{2,\alpha}}{\eta^\alpha} C M_0^\alpha T \theta^\delta$.
Condition \ref{it_boundedness} shows that 
\begin{equation*}
e_4
\leq \frac{5}{\eta} \E \left[ \int_0^T \1_{[\tau_n,\tau_n+\theta]}(s)  \norm{S(\tau_n+\theta-s) F(X_{n-1}(s))} \ds \right] 
\leq \frac{5}{\eta} \theta M_0,
\end{equation*}
and similarly we conclude $e_5 \leq M_0^\alpha c_{2,\alpha} \frac{5^\alpha}{\eta^\alpha} \theta$.
Consequently, the Aldous Condition holds.

Step 3. According to \cite[Cor.\ 3.8.2, Th.\ 6.7.3]{Bergh_Lofstrom} it follows from \ref{it_contractions}(iv) that the embedding of $D((-A)^\delta) \subset H$ is compact.
Thus, the closed ball $\bar B_{D((-A)^\delta)}(0,R)$ is compact in $H$ for each $R>0$. Furthermore, Theorem 
\ref{th_tail_estimate} implies
\begin{align*}
&P \left( X_n(t) \notin \bar B_{D((-A)^\delta)}(0,R) \right) \\
&\leq P \left( \norm{S(t)X_0}_{D((-A)^\delta)} >\frac{R}{3} \right) + P \left( \norm{\int_0^t S(t-s) F(X_{n-1}(s)) \ds}_{D((-A)^\delta)} >\frac{R}{3} \right) \\
&\quad + P \left( \norm{\int_0^t S(t-s) G(X_{n-1}(s-)) \dL(s)}_{D((-A)^\delta)} >\frac{R}{3} \right) \\
&\leq P\left( \norm{S(t)X_0}_{D((-A)^\delta)} >\frac{R}{3} \right) + \frac{3}{R} \E \left[ \int_0^T \norm{S(t-s) F(X_{n-1}(s))}_{D((-A)^\delta)} \ds \right] \\
&\quad + \frac{c_{2,\alpha}3^\alpha}{R^\alpha} \E \left[ \int_0^t \norm{S(t-s)G(X_{n-1}(s-))}_{L_{\rm HS}(U,D((-A)^\delta))}^\alpha \ds \right] \\
&\leq P\left( \norm{S(t)X_0}_{D((-A)^\delta)} >\frac{R}{3} \right) + \frac{3}{R} M_0 T + \frac{c_{2,\alpha}3^\alpha}{R^\alpha} M_0^\alpha T.
\end{align*}
Since $P \left( \norm{S(t)X_0}_{D((-A)^\delta)} >\tfrac{R}{3} \right) \to 0$
as $R\to\infty$ it follows that Condition \eqref{version_of_compact_containment} holds. 
\end{proof}

We need the following strengthening of the last Lemma, which is counter-intuitive as it is well known that even in one dimension, tightness in $D([0,T];\R^2)$ is strictly stronger than the tightness of the components in $D([0,T];\R)$.
The Hilbert space $H\times H$ is equipped with the norm
$\norm{(x,y)} = \sqrt{\norm{x}^2+\norm{y}^2}$ for all $x,y\in H.$

\begin{Proposition}
\label{pro_tightness_for_joint}
If $(X_n)$ is a sequence of $D([0,T];H)$-valued random variables, satisfying the Aldous Condition and \eqref{version_of_compact_containment}, then the sequence $(X_n,X_{n-1})$ is tight in $D([0,T];H\times H)$.
\end{Proposition}

\begin{proof}
Firstly, we check that the Aldous Condition for the joint sequence $(X_n,X_{n-1}) \in H\times H$ follows from the Aldous Condition for the sequence $(X_n)$. Fix $\epsilon,\eta>0$. 
Since $(X_n)$ satisfies the Aldous Condition, there exists $\delta>0$ such that for any sequence of stopping times such that $\tau_n+\delta\leq T$ we have
$$\sup_{n \in \N} \sup_{0<\theta\leq\delta} P\left(\norm{X_n(\tau_n+\theta)-X_n(\tau_n)} \geq \frac{\eta}{\sqrt{2}} \right) \leq \frac{\epsilon}{2}.$$
It follows that 
\begin{align*}
&P(\norm{(X_n,X_{n-1})(\tau_n+\theta)-(X_n,X_{n-1})(\tau_n)} \geq \eta) \\
&\quad \leq P \left(\norm{X_n(\tau_n+\theta)-X_n(\tau_n)} \geq \frac{\eta}{\sqrt{2}}\right) + P \left(\norm{X_{n-1}(\tau_n+\theta)-X_{n-1}(\tau_n)} \geq \frac{\eta}{\sqrt{2}}\right) \leq \epsilon.
\end{align*}
Secondly, since $(X_n)$ satisfies \eqref{version_of_compact_containment} with some subspace $\Gamma$ of $H$, there exists $R>0$ such that
$P(\norm{X_n(t)}_\Gamma \leq \tfrac{R}{\sqrt{2}}) \geq 1-\frac{\epsilon}{2}$
for all $n\in \N$ and $t \in [0,T]\cap \Q$.
It follows that 
\begin{equation*}
P(\norm{(X_n,X_{n-1})(t)}_{\Gamma\times \Gamma} > R) \leq P\left( \norm{X_n(t)}_\Gamma > \tfrac{R}{\sqrt{2}} \right) + P\left(\norm{X_{n-1}(t)}_\Gamma > \tfrac{R}{\sqrt{2}}\right)
= \epsilon,
\end{equation*}
which shows that Condition \eqref{version_of_compact_containment} holds for $(X_n,X_{n-1})$ with the compact subspace $\Gamma\times\Gamma$.
\end{proof}

\subsection{Estimating the norm of the difference between \texorpdfstring{$X_n$}{Xn} and \texorpdfstring{$X_{n-1}$}{Xn-1}}
\label{subsec_estimating_the_norm_of_difference}

\begin{Lemma}
\label{lem_difference_converge_in_probability}
Assume that \ref{it_contractions}--\ref{it_Lipschitz}  are satisfied and let $(X_n)$ be the Picard approximation sequence defined in \eqref{lem_tightness_of_Picard}. 
If 
$T <1 \wedge \big((c_{2,\alpha}\vee c_{2,\alpha}^{1/\alpha}) c_3 \big)^{-\alpha}$,
then  the sequence $(X_n(t)-X_{n-1}(t))_n$ converges to $0$ in probability for each $t \in [0,T]$.
\end{Lemma}

\begin{proof}
Step 1: we prove that for every $t \in [0,T]$  we have
\begin{equation}
\label{convergence_of_difference_in_Lp}
\lim_{p \nearrow \alpha} \lim_{n\to \infty} \E \left[ \norm{X_n(t)-X_{n-1}(t)}^p \right] = 0.
\end{equation}
We fix some $p\in (1,\alpha)$ and define $c_I:= 2^{\alpha-1}\big( \tfrac{\alpha-1}{(c_{2,\alpha}^{1/\alpha} \wedge c_{2,\alpha} ) \alpha} M_0^p + M_0^{p/\alpha} \big)$.  Proposition \ref{pro_moment_inequality_2} and Condition \ref{it_boundedness} imply
\begin{align} \label{bound_on_X_1_and_X_0}
&\E \left[ \norm{X_1(t)-X_0(t)}^p \right]  \notag \\
&\qquad\leq 2^{p-1} \left( \E \left[ \norm{\int_0^t S(t-s) F(X_0) \ds}^p \right] + \E \left[ \norm{\int_0^t S(t-s) G(X_0)\dL(s)}^p \right] \right)\notag \\
&\qquad\leq 2^{p-1} \left( T^p M_0^p + C_{\alpha,p} \left( \E \left[ \int_0^T \HSnorm{S(s)G(X_0)}^\alpha \ds \right] \right)^{p/\alpha} \right) \notag \\
&\qquad \leq 2^{p-1} \left( T^p M_0^p + C_{\alpha,p} T^{p/\alpha} M_0^p \right)\notag \\
&\qquad \leq C_{\alpha,p} c_I, 
\end{align}
where in the last line we use $T<1$ and $c_{2,\alpha}^{1/\alpha}\wedge c_{2,\alpha}\le c_{2,\alpha}^{p/\alpha}$. 
Similarly as in \eqref{moment_of_difference_in_uniqueness}, we conclude from  Corollary \ref{pro_moment_inequality_2} and Condition \ref{it_Holder} that for each $n\in\N$ we have 
$$\E \left[\norm{X_n(t)-X_{n-1}(t)}^p \right]  \\
\leq C_{\alpha,p} c_3 \left( \int_0^t  \E \left[ \norm{X_{n-1}(t_1)-X_{n-2}(t_1)}^p \right] \ud t_1 \right)^{p/\alpha}.$$
By iterating the last inequality and using the upper bound \eqref{bound_on_X_1_and_X_0} we obtain
\begin{align*}
\E & \left[ \norm{X_n(t)-X_{n-1}(t)}^p  \right] \\
&\leq C_{\alpha,p}^{1+\frac{p}{\alpha}} c_3^{1+\frac{p}{\alpha}} \bigg( \int_0^t \left( \int_0^{t_1} \E \left[ \norm{X_{n-2}(t_2)-X_{n-3}(t_2)}^p \right] \ud t_2 \right)^{p/\alpha} \dt_1 \bigg)^{p/\alpha} \\
&\leq C_{\alpha,p}^{1 + \frac{p}{\alpha} + \left( \frac{p}{\alpha} \right)^2+\ldots+\left(\frac{p}{\alpha}\right)^{n-2}} c_3^{1 + \frac{p}{\alpha} + \frac{p^2}{\alpha^2} + \ldots + \left(\frac{p}{\alpha} \right)^{n-2}} \\
&\qquad \times \bigg(\int_0^t \bigg( \int_0^{t_1} \ldots \left(  \int_0^{t_{n-2}} \E \left[ \norm{X_1(t_{n-1})-X_0(t_{n-1})}^p \right] \ud t_{n-1} \right)^{p/\alpha} \ldots \bigg)^{p/\alpha} \dt_1 \bigg)^{p/\alpha} \\
&\leq C_{\alpha,p}^{\frac{1-\left(\frac{p}{\alpha}\right)^n}{1-\frac{p}{\alpha}}} c_3^{\frac{1-\left(\frac{p}{\alpha}\right)^n}{1-\frac{p}{\alpha}}} c_I^{\left(\frac{p}{\alpha}\right)^{n-1}}  \bigg( \int_0^t \bigg( \int_0^{t_1} \ldots \left(  \int_0^{t_{n-2}} \dt_{n-1} \right)^{p/\alpha} \ldots \bigg)^{p/\alpha} \dt_1 \bigg)^{p/\alpha}.
\end{align*}
Calculating the integrals above and inserting the formula for the constant $C_{\alpha,p}$ implies 
\begin{align*}
&\E \left[ \norm{X_n(t)-X_{n-1}(t)}^p \right] \\
&\leq (c_{2,\alpha}^{p/\alpha} \alpha)^{\frac{1-\left(\frac{p}{\alpha}\right)^n}{1-\frac{p}{\alpha}}} c_3^{\frac{1-\left(\frac{p}{\alpha}\right)^n}{1-\frac{p}{\alpha}}} c_I^{\left(\frac{p}{\alpha}\right)^{n-1}} 
\left(\tfrac{1}{\alpha-p}\right)^{1+\left(\frac{p}{\alpha}\right)^{n-1}} 
t^{\tfrac{p}{\alpha} \frac{1-\left(\frac{p}{\alpha}\right)^{n-2}}{1-\frac{p}{\alpha}}} \alpha^{-\frac{p}{\alpha}\frac{1-\left(\frac{p}{\alpha}\right)^{n-2}}{1-\frac{p}{\alpha}}} 
 \prod_{k=2}^{n-1} \left(\tfrac{1}{1-\left(\frac{p}{\alpha}\right)^k} \right)^{\left(\frac{p}{\alpha}\right)^{n-k}} \\
&=: \xi(n,p).
\end{align*}
To establish $\xi(n,p)\to 0$ as $n\to\infty$, we  first show 
\begin{equation}
\label{lim_product_1}
\lim_{n \to \infty} \prod_{k=2}^{n-1} \left(\frac{1}{1-\left(\frac{p}{\alpha}\right)^k} \right)^{\left(\frac{p}{\alpha}\right)^{n-k}} =1.
\end{equation}
Since the left side is larger than or equal to 1,  taking logarithms in \eqref{lim_product_1} shows that it is enough to prove 
\begin{equation}
\label{enough_to_prove_limsup}
\limsup_{n\to \infty } \sum_{k=2}^{n-1} \left( \frac{p}{\alpha} \right)^{n-k} \log\left( \frac{1}{1- \left( \frac{p}{\alpha}\right)^k} \right) \leq 0.
\end{equation}
We split the sum according to $k\leq \left\lfloor \frac{n}{2} \right\rfloor$ and $k> \left\lfloor \frac{n}{2} \right\rfloor$.
Since $\log\left( \frac{1}{1- \left( \frac{p}{\alpha}\right)^k} \right) \leq \log \left( \frac{1}{1-\left( \frac{p}{\alpha}\right)^2} \right)$ we have
\begin{align}
\label{k_small}
\sum_{k=2}^{\left\lfloor \frac{n}{2} \right\rfloor} \left( \frac{p}{\alpha} \right)^{n-k} \log\left( \frac{1}{1- \left( \frac{p}{\alpha}\right)^k} \right)
&\leq \log \left( \frac{1}{1-\left( \frac{p}{\alpha}\right)^2} \right) \sum_{k=2}^{\left\lfloor \frac{n}{2} \right\rfloor} \left(\frac{p}{\alpha}\right)^{n-k} \notag \\
&= \log \left( \frac{1}{1-\left( \frac{p}{\alpha}\right)^2} \right) \left(\frac{p}{\alpha}\right)^{n-\left\lfloor \frac{n}{2} \right\rfloor} \frac{1-\left(\frac{p}{\alpha}\right)^{\left\lfloor \frac{n}{2}\right\rfloor-1}}{1-\frac{p}{\alpha}}  \to 0
\end{align}
as $n\to \infty$. For estimating the sum with $k> \left\lfloor \frac{n}{2} \right\rfloor$  we use the inequality
$$\log \left( \frac1{1-x} \right) = - \log(1-x) \leq \log(4) x \qquad \text{for all } x \in \left(0,\frac{1}{2} \right).$$
If $n$ is large enough such that $\left( \frac{p}{\alpha} \right)^{\left\lfloor \frac{n}{2} \right\rfloor +1} < \frac{1}{2}$ it follows that 
\begin{align}
\label{k_large}
\sum_{k=\left\lfloor \frac{n}{2} \right\rfloor+1}^{n-1} \left( \frac{p}{\alpha} \right)^{n-k} \log\left( \frac{1}{1- \left( \frac{p}{\alpha}\right)^k} \right)
&\leq \log(4) \sum_{k=\left\lfloor \frac{n}{2} \right\rfloor+1}^{n-1} \left( \frac{p}{\alpha} \right)^{n-k} \left( \frac{p}{\alpha} \right)^k \notag \\
&= \log(4) \left(n-1-\left\lfloor \frac{n}{2} \right\rfloor \right) \left(\frac{p}{\alpha} \right)^n \to 0 
\end{align}
as $n\to \infty$. 
Combining \eqref{k_small} and \eqref{k_large} finishes the proof of \eqref{lim_product_1} due to  \eqref{enough_to_prove_limsup}.
It follows that
\begin{align}
\label{limit_of_xi}
\lim_{n\to \infty} \xi(n,p) 
&=  (c_{2,\alpha}^{p/\alpha}\alpha c_3 t^{p/\alpha})^{\frac{1}{1-\frac{p}{\alpha}}} \frac{1}{\alpha-p} \frac{1}{\alpha^{\frac{p}{\alpha} \frac{1}{1-\frac{p}{\alpha}}}} \notag \\
&\leq \left( (c_{2,\alpha}\vee c_{2,\alpha}^{1/\alpha}) c_3 t^{1/\alpha} \right)^{\frac{\alpha}{\alpha-p}} \frac{1}{\alpha-p}.
\end{align}
Since the assumption on $T$ guarantees that $(c_{2,\alpha}\vee c_{2,\alpha}^{1/\alpha}) c_3 t^{1/\alpha}$ is smaller than $1$, the right side of \eqref{limit_of_xi} converges to $0$ as $p\nearrow \alpha$, which completes the 
proof of \eqref{convergence_of_difference_in_Lp}.

Step 2.
For fixed $\epsilon \in (0,1)$, Equation \eqref{convergence_of_difference_in_Lp} implies that there exists $p_0>1$ such that for $p\in [p_0,\alpha)$ we have
$$\lim\limits_{n \to \infty} \E \left[ \norm{X_n(t)-X_{n-1}(t)}^p \right] \leq \frac{\epsilon^{1+\alpha}}{2}.$$ 
It follows that there exists $n_0(p)$ such that for $n\geq n_0(p)$ we have 
$$\E \left[ \norm{X_n(t)-X_{n-1}(t)}^p \right] \leq \epsilon^{1+\alpha}.$$  
Markov's inequality implies for $n \geq n_0(p)$ that 
\begin{equation*}
P(\norm{X_n(t)-X_{n-1}(t)} \geq \epsilon) 
\leq \frac{1}{\epsilon^p} \E \left[ \norm{X_n(t)-X_{n-1}(t)}^p \right] 
\leq \epsilon,
\end{equation*}
which completes the proof.
\end{proof}

\subsection{Proof of the main result}
\label{subsec_proof_main_result}

\begin{proof}[Proof of Theorem \ref{th_existence_weak_sol_main_result}]
We divide the proof in several small steps. 

(1) Lemma \ref{lem_tightness_of_Picard} and Proposition \ref{pro_tightness_for_joint} imply that the joint sequence $(X_n,X_{n-1})$ is tight. Consequently,  with $A_L$ defined in \eqref{definition_A_L}, the sequence $(X_n,X_{n-1},A_L)$ is tight in the space $D([0,T];H \times H \times \R^\infty)$
and thus there exists a convergent subsequence $(X_{n_k},X_{n_k-1},A_L)$ due to Prokhorov's theorem.
The Skorokhod theorem \cite[Th.\ 6.7]{Billingsley} implies that there exists a probability basis $(\bar{\Omega},\bar{\mathcal{F}},\bar{P},(\bar{\mathcal{F}}_t))$, a sequence $(\bar{X}_k,\bar{Y}_{k},\bar{B}_k)$ and random variables $\bar{X}, \bar{Y}$ with the same law:
\begin{equation} \label{eq.Skorokhod-equal}
\mathcal{L}(X_{n_k},X_{n_k-1},A_L) = \mathcal{L}(\bar{X}_k,\bar{Y}_{k},\bar{B}_k)\qquad\text{for all } k \in \N,
\end{equation}
and such that $(\bar{X}_{k},\bar{Y}_{k},\bar{B}_k) \to (\bar{X},\bar{Y},\bar{B})$ a.s.\ as $k\to \infty$.

(2) Recall that $A_L$ in \eqref{definition_A_L} is defined in terms of a dense set $D:=\{u_n : n \in \N\}\subseteq U$. As in Lemma \ref{lem_metric_space}, but using $\bar{B}_k$ and $\bar{B}$ we can construct cylindrical L\'evy processes $(\bar{L}_k(t):\, t\ge 0)$ and  $(\bar{L}(t):\, t\ge 0)$ with cylindrical random variables  $\bar{L}_k(t), \bar{L}(t) \colon U \to L^0(\bar{\Omega},\bar{\mathcal{F}},\bar{P})$ and such that $L$, $\bar{L}_k$ and $\bar{L}$ have the same cylindrical distributions.  We claim that $\bar{L}_k(t)u$ converges in probability to $\bar{L}(t)u$ for all $t\ge 0$ and $u\in U$. If $u\in D$, then $ \bar{L}_k(\cdot)u$ converges a.s.\ to $\bar{L}(\cdot)u$ in $D([0,T];\R)$ because $\bar{B}_k$ converges a.s.\ to $B$ in $D([0,T];\R^\infty)$. Thus, since 
$\bar{L}(\cdot)u$ does not have any fixed time of discontinuity, it follows that  $\bar L_k(t)u \to \bar L(t)u$ for every $t \in [0,T]$, see \cite[Th.\ 12.5]{Billingsley}. If $u\in U\setminus D$, then fix $\epsilon \in (0,\tfrac{1}{3})$. 
Choose $n \in \N$ such that $c_{1,\alpha} t \norm{u-u_n}^\alpha \left( \frac{\epsilon}{3} \right)^\alpha \leq \frac{\epsilon}{3}$ and let $k_0$ be such that $\bar{P} \left( \abs{ \bar L_k(t) u_n - \bar L(t)u_n} \geq \frac{\epsilon}{3} \right) \leq \frac{\epsilon}{3}$ for all $k \geq k_0$. 
By Lemma \ref{lem_tail_estimate_radonified} (under the identification $u \in L_{\rm HS}(U,\R)$ with $\norm{u} = \norm{u}_{L_{\rm HS}(U,H)}$) we have for all $k \geq k_0$
\begin{align*}
&\bar{P} \left( \abs{\bar L_k(t)u - \bar L(t)u} \geq \epsilon \right) \\
&\leq \bar{P} \left( \abs{\bar L_k(t)u - \bar L_k(t) u_n} \geq \frac{\epsilon}{3} \right) + \bar{P} \left( \abs{ \bar L_k(t) u_n - \bar L(t)u_n} \geq \frac{\epsilon}{3} \right) + \bar{P} \left( \abs{\bar L(t) u_n - \bar L(t)u} \geq \frac{\epsilon}{3} \right) \\
&\leq c_{1,\alpha} t \norm{u-u_n}^\alpha \left( \frac{\epsilon}{3} \right)^\alpha + \bar{P} \left( \abs{ \bar L_k(t) u_n - \bar L(t)u_n} \geq \frac{\epsilon}{3} \right) + c_{1,\alpha} t \norm{u-u_n}^\alpha \left( \frac{\epsilon}{3} \right)^\alpha \\
& \leq \epsilon, 
\end{align*}
which shows that $\bar{L}_k(t)u$ converges in probability to $\bar{L}(t)u$ for all $t\ge 0$ and $u\in U$. 

(3) We prove that the processes $\bar{X}$ and $\bar{Y}$ are indistinguishable if the time interval $[0,T]$ is sufficiently small. Since $(\bar{X}_k(t),\bar{Y}_k(t))$ has the same law as $(X_{n_k}(t),X_{n_k-1}(t))$ for each  $t\in [0,T]$,  
 Lemma \ref{lem_difference_converge_in_probability} implies for any $\epsilon>0$ and $t\in [0,T]$ that
$$\bar{P}\left( \norm{\bar{X}_k(t)-\bar{Y}_k(t)} \geq \epsilon \right)
= P\left( \norm{X_{n_k}(t)-X_{n_k-1}(t)} \geq \epsilon \right) \to 0
\qquad\text{as }k\to\infty. $$
It follows that $\bar{X}_k(t)-\bar{Y}_k(t)$ converges to $0$ in probability. 
By the continuous mapping theorem $\bar{X}_k(t) \to \bar{X}(t)$ a.s.\ and $\bar{Y}_k(t)\to \bar{Y}(t)$ a.s.\ for $t \in [0,T] \setminus (\mathcal{J}(\bar{X})\cup \mathcal{J}(\bar{Y}))$.
By the uniqueness of limits (in probability) we obtain that $\bar{X}(t) = \bar{Y}(t)$ a.s.\ for $t \in [0,T] \setminus (\mathcal{J}(\bar{X})\cup \mathcal{J}(\bar{Y}))$.
Using the fact that $[0,T] \setminus (\mathcal{J}(\bar{X})\cup \mathcal{J}(\bar{Y}))$ is dense in $[0,T]$ and that the processes $\bar{X}$ and $\bar{Y}$ are c\`adl\`ag we get that they are indistinguishable.

(4) We establish that the assumption \eqref{eq.joint-distribution-equal} in Lemma \ref{lem_unique_distribution_of_integral} is satisfied. 
Let $\Psi_k(s) = S(t-s) G(X_{n_k-1}(s-))$ and $\bar \Psi_k(s) = S(t-s) G(\bar Y_k(s-))$ for $s \in [0,t]$. 
Fix $s_1,\ldots, s_m,t_1,\ldots, t_n \in [0,t]$ and $v_1,\ldots, v_n \in U$.
If  $\{v_1,\ldots, v_n\} \subset D$, then \eqref{eq.Skorokhod-equal} implies
\begin{multline}
\label{vectors_equal_in_dist}
\left( \Psi_k(s_1), \ldots, \Psi_k(s_m), L(t_1)v_1, \ldots, L(t_n)v_n \right)  \\
\stackrel{\mathcal D}{=} \left( \bar \Psi_k(s_1), \ldots, \bar \Psi_k(s_m), \bar L_k(t_1)v_1, \ldots, \bar L_k(s_n)v_n \right).
\end{multline}
Otherwise, each $v_i\in U\setminus D$ can be approximated by elements of $\{u_n : n \in \N\}$ and a simple argument using characteristic functions shows that \eqref{vectors_equal_in_dist} holds 
for arbitrary $v_1,\dots, v_n\in U$.  

(5) In this step, we show that $\bar{X}$ is a weak mild solution if the time interval $[0,T]$ is sufficiently small.  Step (4) enables us to apply Lemma \ref{lem_unique_distribution_of_integral}, which yields by the very definition of $X_{n_k}$ in \eqref{Picard_teration}  that
\begin{equation}
\label{new_processes_old_equation}
\bar{X}_k(t) = S(t)\bar{X}_k(0) + \int_0^t S(t-s)F(\bar{Y}_k(s))\ds + \int_0^t S(t-s) G(\bar{Y}_k(s-))\, \ud \bar{L}_k(s).
\end{equation}
For passing to the limit as $k\to\infty$ we make the following observations: 
firstly, the mapping $\pi_0$ defined on $D([0,T];H)$ by $x \mapsto x(0)$ is continuous (see \cite[p.\ 133]{Billingsley}), and thus $\bar{X}_k(0) \to \bar{X}(0)$ a.s. 
Secondly, as $F$ is continuous, it follows from Step (3) that  $S(t-s)F(\bar{Y}_k(s)) \to S(t-s) F(\bar{X}(s))$ a.s.\ for every $s\in [0,t] \setminus \mathcal{J}(\bar{X})$.
The boundedness of $F$ according to \ref{it_boundedness} implies by Lebesgue dominated convergence theorem that
\begin{equation*}
\int_0^t S(t-s)F(\bar{Y}_k(s-)) \ds \to \int_0^t S(t-s) F(\bar{X}(s-)) \ds,
\end{equation*}
in $L^1(\Omega;H)$ for every $t\in [0,T]$.
Thirdly, we consider the convergence of the stochastic integrals.
By Lemma \ref{lem_continuity_in_Skorokhod_space} we have that $S(t-s)G(\bar{Y}_k(s-))$ converges to $S(t-s)G(\bar{Y}(s-))$ for almost all $s \in [0,t]$. Due to \ref{it_boundedness}, it follows from Lebesgue's dominated convergence theorem and Step (3)   that
$$\lim_{k\to \infty} \E \left[ \int_0^t \norm{S(t-s)G(\bar{Y}_k(s-)) - S(t-s) G(\bar{X}(s-))}_{L_{\rm HS}(U,H)}^\alpha \ds \right] = 0.$$
Lemma \ref{lem_general_integrand} implies that 
\begin{equation*}
\lim_{k\to\infty} \int_0^t S(t-s)G(\bar{Y}_k(s-)) \, \ud \bar{L}_k(s) = \int_0^t S(t-s) G(\bar{X}(s-)) \, \ud \bar{L}(s)
\end{equation*}
in probability for every $t\in [0,T]$. In summary, we may pass to the limit in \eqref{new_processes_old_equation}
as $k\to \infty$ to obtain for each $t\in [0,T]\setminus \mathcal{J}(\bar{X})$ that
$$\bar{X}(t) = S(t)\bar{X}(0) +  \int_0^t S(t-s)F(\bar{X}(s))\ds +  \int_0^t S(t-s)G(\bar{X}(s-)) \, \ud \bar{L}(s).$$
almost surely for $t \notin \mathcal{J}(\bar{X})$.
The right-hand side is  c\`adl\`ag by the dilation theorem \cite[Th.\ I.8.1]{Foias_Nagy} and the left-hand side is c\`adl\`ag by definition. Thus the conclusion holds for any $t\in [0,T]$.
This proves existence of a weak mild solution under the additional assumption of sufficiently small $T$.

(6)  Since the solution of \eqref{SPDE_stable} is  pathwise uniqueness for sufficiently small interval $[0,T]$ according to Proposition \ref{pro_pathwise_uniqueness}, Theorem 1.5 and Proposition 2.13 in \cite{Kurtz_Yamada} imply the existence of a strong solution on $[0,T]$. The details on the applicability of the abstract setting of \cite{Kurtz_Yamada} are given in the PhD thesis \cite{Tomasz_PhD}.

(7) 
Knowing the existence of a strong solution on a small interval enables us to work on the original probability space for arbitrary $T>0$: choose a time $T_0$ and $n\in\N$ satisfying the conditions in Lemma \ref{pro_pathwise_uniqueness} and  \ref{lem_difference_converge_in_probability} such that $[0,T_0]\cup[T_0,2T_0]\cup \ldots \cup [(n-1)T_0,nT_0] = [0,T]$. The previous considerations show the existence of a unique mild solution $X_1$ on $[0,T_0]$. Similarly there exists a unique mild solution  $X_2$ on $[T_0,2T_0]$ with the initial condition $X_1(T_0)$. We continue this procedure and define $X(t):=X_k(t)$ for $t \in [(k-1)T_0,kT_0]$ for $k=2,\ldots, n$.
\end{proof}

\bibliography{References_stable}

\begin{thebibliography}{10}

\bibitem{Applebaum_Riedle}
D.~Applebaum and M.~Riedle.
\newblock Cylindrical {L}\'evy processes in {B}anach spaces.
\newblock {\em Proc. Lond. Math. Soc.}, 101(3):697--726, 2010.

\bibitem{Balan}
R.~M. Balan.
\newblock S{PDE}s with {$\alpha$}-stable {L}\'evy noise: a random field
  approach.
\newblock {\em Int. J. Stoch. Anal.}, 2014.

\bibitem{Bergh_Lofstrom}
J.~Bergh and J.~L\"{o}fstr\"{o}m.
\newblock {\em Interpolation spaces. {A}n introduction}.
\newblock Berlin-New York: Springer-Verlag, 1976.

\bibitem{Billingsley}
P.~Billingsley.
\newblock {\em Convergence of probability measures}.
\newblock New York: John Wiley \& Sons, Inc., second edition, 1999.
\newblock A Wiley-Interscience Publication.

\bibitem{Brzezniak_Hausenblas_uniqueness}
Z.~Brze\'{z}niak and E.~Hausenblas.
\newblock Uniqueness in law of the {I}t\^{o} integral with respect to
  {L}\'{e}vy noise.
\newblock In {\em Seminar on {S}tochastic {A}nalysis, {R}andom {F}ields and
  {A}pplications {VI}}, volume~63 of {\em Progr. Probab.}, pages 37--57.
  Birkh\"{a}user/Springer Basel AG, Basel, 2011.

\bibitem{Brzezniak_Zabczyk}
Z.~Brze\'zniak and J.~Zabczyk.
\newblock Regularity of {O}rnstein-{U}hlenbeck processes driven by a {L}\'evy
  white noise.
\newblock {\em Potential Anal.}, 32(2):153--188, 2010.

\bibitem{Chong_stochastic_PDEs}
C.~Chong.
\newblock Stochastic {PDE}s with heavy-tailed noise.
\newblock {\em Stochastic Process. Appl.}, 127(7):2262--2280, 2017.

\bibitem{Chong_Dalang_Humeau}
C.~Chong, R.~C. Dalang, and T.~Humeau.
\newblock Path properties of the solution to the stochastic heat equation with
  {L}{\'e}vy noise.
\newblock {\em Stochastics and Partial Differential Equations: Analysis and
  Computations}, Aug 2018.

\bibitem{Da_Prato_Zabczyk}
G.~Da~Prato and J.~Zabczyk.
\newblock {\em Stochastic equations in infinite dimensions}.
\newblock Cambridge: Cambridge University Press, second edition, 2014.

\bibitem{Diestel}
J.~Diestel, H.~Jarchow, and A.~Tonge.
\newblock {\em Absolutely summing operators}.
\newblock Cambridge: Cambridge University Press, 1995.

\bibitem{Dieudonne}
J.~Dieudonn{\'e}.
\newblock {\em Foundations of Modern Analysis, Vol. 1}.
\newblock New York: Academic Press, Inc., 1969.

\bibitem{Engelking}
R.~Engelking.
\newblock {\em General topology}.
\newblock Berlin: Heldermann Verlag, second edition, 1989.
\newblock Translated from the Polish by the author.

\bibitem{Ethier_Kurtz}
S.~N. Ethier and T.~G. Kurtz.
\newblock {\em Markov processes. Characterization and convergence}.
\newblock New York: John Wiley \& Sons, Inc., 1986.

\bibitem{Gelfand_Vilenkin}
I.~M. Gel'fand and N.~Ya. Vilenkin.
\newblock {\em Generalized functions. {V}ol. 4: {A}pplications of harmonic
  analysis}.
\newblock Translated by Amiel Feinstein. Academic Press, New York - London,
  1964, 1964.

\bibitem{Gine_Marcus}
E.~Gin\'e and M.~B. Marcus.
\newblock The central limit theorem for stochastic integrals with respect to
  {L}\'evy processes.
\newblock {\em Ann. Probab.}, 11(1):58--77, 1983.

\bibitem{Hausenblas}
E.~Hausenblas.
\newblock Existence, uniqueness and regularity of parabolic {SPDE}s driven by
  {P}oisson random measure.
\newblock {\em Electron. J. Probab.}, 10:1496--1546, 2005.

\bibitem{Jacod_Shiryaev}
J.~Jacod and A.~N. Shiryaev.
\newblock {\em Limit theorems for stochastic processes}.
\newblock Berlin: Springer-Verlag, second edition, 2003.

\bibitem{Jakubowski_1986}
A.~Jakubowski.
\newblock On the {S}korokhod topology.
\newblock {\em Ann. Inst. H. Poincar\'{e} Probab. Statist.}, 22(3):263--285,
  1986.

\bibitem{Jakubowski_1988}
A.~Jakubowski.
\newblock Tightness criteria for random measures with application to the
  principle of conditioning in {H}ilbert spaces.
\newblock {\em Probab. Math. Statist.}, 9(1):95--114, 1988.

\bibitem{Jakubowski_Riedle}
A.~Jakubowski and M.~Riedle.
\newblock Stochastic integration with respect to cylindrical {L}\'evy
  processes.
\newblock {\em Ann. Probab.}, 45(6B):4273--4306, 2017.

\bibitem{Jentzen_Rockner}
A.~Jentzen and M.~R\"ockner.
\newblock Regularity analysis for stochastic partial differential equations
  with nonlinear multiplicative trace class noise.
\newblock {\em J. Differential Equations}, 252(1):114--136, 2012.

\bibitem{Tomasz_PhD}
T.~Kosmala.
\newblock {\em Stochastic Partial Differential Equations Driven by Cylindrical
  {L}\'evy Processes}.
\newblock PhD thesis, King's College London, 2020.
\newblock URL:
  \url{https://kclpure.kcl.ac.uk/portal/files/155447870/2021_Kosmala_Tomasz_1568234_ethesis.pdf}.

\bibitem{Kurtz_Yamada}
T.~G. Kurtz.
\newblock The {Y}amada-{W}atanabe-{E}ngelbert theorem for general stochastic
  equations and inequalities.
\newblock {\em Electron. J. Probab.}, 12:951--965, 2007.

\bibitem{Linde}
W.~Linde.
\newblock {\em Probability in {B}anach spaces -- stable and infinitely
  divisible distributions}.
\newblock Chichester: John Wiley \& Sons, Ltd., second edition, 1986.

\bibitem{Lototsky_Rozovsky}
S.~V. Lototsky and B.~L. Rozovsky.
\newblock {\em Stochastic partial differential equations}.
\newblock Universitext. Springer, Cham, 2017.

\bibitem{Lunardi}
A.~Lunardi.
\newblock {\em Interpolation theory}.
\newblock Pisa: Edizioni della Normale, 2018.
\newblock Third edition.

\bibitem{Motyl_2013}
E.~Motyl.
\newblock Stochastic {N}avier-{S}tokes equations driven by {L}\'evy noise in
  unbounded 3{D} domains.
\newblock {\em Potential Anal.}, 38(3):863--912, 2013.

\bibitem{Mueller}
C.~Mueller.
\newblock The heat equation with {L}\'{e}vy noise.
\newblock {\em Stochastic Process. Appl.}, 74(1):67--82, 1998.

\bibitem{Mytnik}
L.~Mytnik.
\newblock Stochastic partial differential equation driven by stable noise.
\newblock {\em Probab. Theory Related Fields}, 123(2):157--201, 2002.

\bibitem{Pazy}
A.~Pazy.
\newblock {\em Semigroups of linear operators and applications to partial
  differential equations}, volume~44.
\newblock New York: Springer-Verlag, 1983.

\bibitem{Peszat_Zabczyk}
S.~Peszat and J.~Zabczyk.
\newblock {\em Stochastic partial differential equations with L\'evy noise. An
  evolution equation approach}.
\newblock Cambridge: Cambridge University Press, 2007.

\bibitem{Prevot_Rockner}
C.~Pr\'ev\^{o}t and M.~R\"ockner.
\newblock {\em {A concise course on stochastic partial differential
  equations.}}
\newblock Berlin: Springer, 2007.

\bibitem{Riedle_infinitely}
M.~Riedle.
\newblock Infinitely divisible cylindrical measures on {B}anach spaces.
\newblock {\em Studia Math.}, 207(3):235--256, 2011.

\bibitem{Riedle_stable}
M.~Riedle.
\newblock Stable cylindrical {L}\'{e}vy processes and the stochastic {C}auchy
  problem.
\newblock {\em Electron. Commun. Probab.}, 23:Paper No. 36, 12, 2018.

\bibitem{Rosinski_Woyczynski_moment}
J.~Rosi\'nski and W.~A. Woyczy\'nski.
\newblock Moment inequalities for real and vector {$p$}-stable stochastic
  integrals.
\newblock In {\em Probability in {B}anach spaces, {V} ({M}edford, {M}ass.,
  1984)}, volume 1153 of {\em Lecture Notes in Math.}, pages 369--386. Berlin:
  Springer, 1985.

\bibitem{Schwartz-Radon}
L.~{Schwartz}.
\newblock {Radon measures on arbitrary topological spaces and cylindrical
  measures.}
\newblock {Tata Institute of Fundamental Research Studies in Mathematics. 6.
  London: Oxford University Press, published for the Tata Institute of
  Fundamental Research. }, 1973.

\bibitem{Segal}
I.~E. Segal.
\newblock Abstract probability spaces and a theorem of {K}olmogoroff.
\newblock {\em Amer. J. Math.}, 76:721--732, 1954.

\bibitem{Foias_Nagy}
B.~Sz.-Nagy and C.~Foia\c{s}.
\newblock {\em Harmonic analysis of operators on {H}ilbert space}.
\newblock Budapest: Akad\'{e}miai Kiad\'{o}, and Amsterdam-London:
  North-Holland Publishing Co., 1970.
\newblock Translated from the French and revised.

\bibitem{Willett_Wong}
D.~Willett and J.~S.~W. Wong.
\newblock On the discrete analogues of some generalizations of {G}ronwall's
  inequality.
\newblock {\em Monatsh. Math.}, 69:362--367, 1965.

\bibitem{Xiong_Yang}
J.~Xiong and X.~Yang.
\newblock Existence and pathwise uniqueness to an {SPDE} driven by
  {$\alpha$}-stable colored noise.
\newblock {\em Stochastic Process. Appl.}, 129(8):2681--2722, 2019.

\bibitem{Yang_Zhou}
X.~Yang and X.~Zhou.
\newblock Pathwise uniqueness for an {SPDE} with {H}\"{o}lder continuous
  coefficient driven by {$\alpha$}-stable noise.
\newblock {\em Electron. J. Probab.}, 22:Paper No. 4, 48, 2017.

\end{thebibliography}
\bibliographystyle{plainurl}

\end{document}